\documentclass[10pt, leqno]{amsart}

\usepackage{graphicx}
\usepackage{amssymb,amsmath,amsthm,amscd}
\usepackage{mathrsfs}

\usepackage{lscape}
\usepackage{enumerate}
\usepackage[usenames,dvipsnames]{color}
\usepackage[colorlinks=true, pdfstartview=FitV, linkcolor=blue,citecolor=blue,urlcolor=blue]{hyperref}
 \usepackage{graphicx}

\usepackage{tikz}
\usepackage[all]{xy}
\usepackage{color}
\usepackage[vcentermath,enableskew]{youngtab}

\allowdisplaybreaks[3]

\numberwithin{equation}{section}

\theoremstyle{plain}
\newtheorem{lemma}{Lemma}[section]
\newtheorem{proposition}[lemma]{Proposition}
\newtheorem{theorem}[lemma]{Theorem}
\newtheorem{algorithm}[lemma]{Algorithm}

\newtheorem{corollary}[lemma]{Corollary}
\newtheorem{conjecture}{Conjecture}

\theoremstyle{definition}
\newtheorem{remark}[lemma]{Remark}
\newtheorem{example}[lemma]{Example}

\newtheorem{definition}[lemma]{Definition}

\newtheorem{convention}[lemma]{Convention}

\newcommand{\srt}[2]{{}_{\langle #1,#2 \rangle}}
\newcommand{\hd}{{\operatorname{hd}}}
\newcommand{\soc}{{\operatorname{soc}}}
\newcommand{\dct}[2]{\overset{#2}{\underset{#1}{\conv}} }

\newcommand{\g}{\mathfrak{g}}

\newcommand{\C}{\mathbb{C}}

\newcommand{\Z}{\mathbb{Z}}
\newcommand{\N}{\mathbb{N}}
\newcommand{\cor}{\mathbf{k}}

\newcommand{\h}{\mathsf{h}}
\newcommand{\W}{\mathsf{W}}
\newcommand{\seteq}{\mathbin{:=}}

\newcommand{\al}{\alpha}
\newcommand{\la}{\lambda}

\newcommand{\be}{\beta}

\newcommand{\PR}{\Phi^+}
\newcommand{\tw}{{\widetilde{w}}}

\newcommand{\redex}{{\widetilde{w}}}

\newcommand{\redez}{{\widetilde{w}_0}}
\newcommand{\um}{\underline{m}}

\newcommand{\tb}{\mathtt{b}}
\newcommand{\rl}{\mathsf{Q}}
\newcommand{\wl}{\mathsf{P}}
\newcommand{\cmA}{\mathsf{A}}

\newcommand{\wt}{{\rm wt}}

\newcommand{\het}{{\rm ht}}
\newcommand{\Rep}{{\rm Rep}}
\newcommand{\Dim}{\underline{{\rm dim}}}

\newcommand{\lf}{[\hspace{-0.3ex}[}
\newcommand{\rf}{]\hspace{-0.3ex}]}
\newcommand{\shp}{\hspace{-0.4ex}+\hspace{-0.4ex}}
\newcommand{\soplus}{\mathop{\mbox{\normalsize$\bigoplus$}}\limits}
\newcommand{\conv}{\mathbin{\mbox{\large $\circ$}}}
\newcommand{\Stom}{\overset{\to}{S}_{\redez}(\um)}
\newcommand{\Sgetsm}{\overset{\gets}{S}_{\redez}(\um)}
\newcommand{\Lto}{\longrightarrow}

\newlength{\mylength}
\newcommand{\Aut}{{\mathrm{Aut}}}
\newcommand{\Ind}{{\mathrm{Ind}}}
\newcommand{\Mod}{{\mathrm{Mod}}}

\newcommand{\rmat}[1]{{\mathbf r}_{\mspace{-2mu}\raisebox{-.5ex}{${\scriptstyle{#1}}$}}}

\title[Combinatorial AR-quivers and reduced expressions] {Combinatorial Auslander-Reiten quivers and reduced expressions}

\author[S.-j. Oh, U. Suh]{Se-jin Oh$^{\dagger}$, UhiRinn Suh$^\ddagger$}

\address{Department of Mathematics, Ewha Womans University, Seoul 120-750, South Korea}
\email{sejin092@gmail.com}
\address{Department of Mathematical Sciences, KAIST, 291 Daehak-ro Yuseonggu
Daejeon, 305-701 South Korea}
\email{uhrisu@gmail.com}

\thanks{$^{\dagger}$This work was supported by NRF Grant \#2016R1C1B2013135.}
\thanks{$^\ddagger$ This work was supported by NRF Grant \#2016R1C1B1010721.}

\keywords{combinatorial AR-quiver, reduced expressions}

\date{\today}

\begin{document}


\begin{abstract}
In this paper, we introduce the notion of combinatorial Auslander-Reiten(AR) quiver for commutation classes $[\redex]$
of $w$ in finite Weyl group. This combinatorial object visualizes the convex partial order $\prec_{[\redex]}$ on the subset $\Phi(w)$ of 
positive roots. By analyzing properties of the combinatorial AR-quivers with labelings and reflection maps, we can apply their properties to the representation theory of KLR algebras and
multiplication structure of dual PBW generators associated to any commutation class $[\redez]$ of the longest element $w_0$.
\end{abstract}

\maketitle

\section*{Introduction}

For a Dynkin quiver $Q$ of finite type ADE, the Auslander-Reiten quiver $\Gamma_Q$ encodes the representation theory of
the path algebra $\C Q$ in the following sense: {\rm (i)} the set of vertices corresponds to the set $\Ind \, Q$ of isomorphism classes  of indecomposable $\C Q$-modules, {\rm (ii)}
the set of arrows corresponds to the set of irreducible morphisms for $M,N \in \Ind \, Q$. On the other hand, by reading the residues of vertices of $\Gamma_Q$ in a {\it compatible way} (\cite{B99}), one can obtain
reduced expressions $\redez$ of the longest element $w_0$ in the Weyl group $\W$. Such reduced expressions can be grouped into one class $[Q]$ via commutation equivalence $\sim$:
$$ \redez \sim \redez \text{ if and only if $\redez'$ can be obtained by applying the short braid relations $s_is_j=s_js_i$ properly. }  $$
A reduced expression in $[Q]$ is called {\it adapted to $Q$}. Reduced expressions in $[Q]$ have been used in representation theory intensively. For example, \cite{HL11,Lus90,Lus97} to name a few. However, there are many reduced expressions of $w_0$ which are not adapted to {\it any} Dynkin quiver $Q$.

Another important role of $\Gamma_Q$ in Lie theory is a realization of the convex partial order $\prec_{Q}$ on $\Phi^+$,
which is defined as follows: For a reduced expression $\redez=s_{i_1}s_{i_2} \cdots s_{i_N} \in [Q]$, we label a positive root $s_{i_1}s_{i_2} \cdots s_{i_{k-1}}\alpha_k \in \Phi^+$
 as $\be_k^\redez$ and assign the {\it residue} $i_k$ to $\be_k^\redez$. Then each reduced expression $\redez\in[Q]$ induces the total order  $<_{\redez}$ on $\Phi^+$, $\be^\redez_k <_\redez \be^\redez_l \iff k < l$. Using the total orders $<_{\redez'}$ for $\redez'\in [Q]$, we obtain $\prec_{Q}$ on $\Phi^+$ as follows:
\begin{itemize}
\item $\alpha \prec_{Q} \beta$ if and only if $\alpha <_{\redez'} \beta$ for all $\redez' \in [Q]$,
\item $\alpha \prec_{Q} \beta$ and $\gamma=\alpha+\beta \in \Phi^+$ imply $\alpha \prec_{Q} \gamma \prec_Q \beta$ (the convexity).
\end{itemize}
Note that it is proved in \cite{Papi94,Zelo87} that any convex total order $<$ on $\Phi^+$ is $<_{\redez}$ for some $\redez$ of $w_0$.
As the definition itself, $\prec_{Q}$ is quite complicated since there are lots of reduced expressions in each $[Q]$. Interestingly, $\Gamma_Q$ realizes $\prec_{Q}$ in the sense that
$$ \text{ $\alpha\prec_{Q} \beta$ if and only if there exists a path from $\beta$ to $\alpha$ in $\Gamma_Q$.} $$

For non-adapted reduced expressions $\redez$ and their commutation classes $[\redez]$, the definition $\prec_{[\redez]}$ is still applicable.
However, there was no study on the order $\prec_{[\redez]}$ on $\Phi^+$ for non-adapted $[\redez]$ and apply them to the representation theory,
to the best knowledge of authors.

In this paper, we introduce new quiver $\Upsilon_{[\redex]}$, named as the {\it combinatorial AR-quiver}, for each reduced expression $\redex$ of $w \in \W$. It realizes
the convex partial order $\prec_{[\redex]}$ on $\Phi(w)$ (Theorem \ref{Thm:order_path}); that is,
$$ \text{ $\alpha\prec_{[\redex]} \beta$ if and only if there exists a path from $\beta$ to $\alpha$ in $\Upsilon_{[\redex]}$.} $$
and hence it can be understood as the Hasse quiver associated to the order $\prec_{[\redex]}$ (\cite{Birkhoff}). We first
suggest an algorithm for constructing $\Upsilon_{[\redex]}$ of $\redex=s_{i_1}\cdots s_{i_\ell}$ (Algorithm \ref{Alg_AbsAR}). If we use {\it residues} as labels of the vertices  $\Upsilon_{[\redex]}^0$  in $\Upsilon_{[\redex]}$ instead of $\Phi(w) \subset \Phi^+$, one can
construct $\Upsilon_{[\redex]}$ instantly. Then we can prove that  $\Upsilon_{[\redex]}=\Upsilon_{[\redex']}$ if and only if $[\redex]=[\redex']$, and hence any reduced expressions in $[\redex]$
can be obtained by reading residues of $\Upsilon_{[\redex]}$ in a {\it compatible way} (Theorem \ref{Thm: compatible reading}).

As one can expect, $\Upsilon_{[Q]}$ is isomorphic to $\Gamma_Q$ so that  $\Upsilon_{[\redex]}$
is understood as a generalizations of $\Gamma_Q$ (Theorem \ref{thm: Upsilon Q Gamma Q}).  Hence combinatorial AR quivers have analogous properties to original AR quivers.  For instance, we prove that (i) an arrow $\alpha\to\beta$ in
$\Upsilon_{[\redex]}$ implies $-(\alpha,\beta)=(\al_{i},\al_j)>0$ where $i$ and $j$ are residues of $\al$ and $\beta$, (ii) positive roots in a {\it sectional path} of $\Upsilon_{[\redex]}$ (Definition \ref{Def:sectional})
satisfy distinguished property with respect to the non-degenerate pairing $( \ , \ )$ on the set of root lattice.

 There are enormous number of reduced expressions for $w_0$ and, by grouping into commutation equivalence classes, we reduce the efforts to understand all reduced expressions. However, there are still too many
commutation classes of reduced expressions so that we consider another equivalence relation called {\it reflection equivalence relations} between commutation equivalence classes.  
A family of equivalence classes induced from the reflection equivalences is called an $r$-cluster point $\lf \redez \rf$.
 We would like to point out that there are lots of similarities between representation theories related to commutation classes $[Q]$ and $[Q']$ in the $r$-cluster point $\lf Q \rf$ (for example, \cite{HL11,Lus90,Oh14A,Oh14D,Oh15E}, see also Corollary \ref{cor: similarity}).
In addition, we introduce the notion of Coxeter composition (Definition \ref{def: ref equi}) with respect to a non-trivial Dynkin diagram automorphism $\sigma$. Our conjecture
is that Coxeter compositions classify $r$-cluster points (Conjecture \ref{Conjecture}).

 The most useful for applications but complicate part in combinatorial AR quivers is computing labels in terms of positive roots.  One can see in Algorithm \ref{Alg_AbsAR} that the labeling of  $\Upsilon_{[\redex]}^0$  with $\Phi(w)$ needs lots of computations. In Section \ref{Sec: Labeling}, we suggest an efficient way to reduce large amount of
computations in general. 
Roughly speaking, every positive root in each sectional path shares a  {\it component} (Definition \ref{Def:comp}).  Hence the labeling for a given vertex follows from joining information of sectional paths
passing through it.

In Section \ref{sec: KLR}, we apply our observations in previous sections to the representation theory of KLR-algebras (\cite{BKM12,HL11,KKK13b,Kato12,Mc12}) which categorifies each dual PBW-basis $\{P_{\redez}(\beta) \ | \ \beta \in \Phi^+\}$ (\cite{Lus90A,Saito94}) associated to the reduced expression $\redez$ of $w_0$.
Using the $\prec_{[\redez]}$ realized by $\Upsilon_{[\redez]}$, we can prove that {\it the proper standard module} $S_{\redez}(\beta)$ (\cite{BKM12,Mc12})
over KLR-algebra of each finite type {\it depends only on} its commutation class $[\redez]$ (up to $q^{\Z}$)
and hence so does the dual PBW-monomial associated to $\redez$ (Theorem \ref{thm: app KLR1}). Note that such property was observed in \cite{Oh15E} (see also \cite{Kato12} for ADE cases).
Here we give an alternative proof.
Furthermore, we prove that the proper standard modules $S_{[\redez]}(\beta)$'s for $\beta$'s lying on the same sectional path {\it $q$-commute} to each other and hence so does
the dual PBW-generators $P_{[\redez]}(\beta)$'s (Proposition \ref{prop: q-comm comparable}). Using the reflection maps on each $\lf \redez \rf$, we also observe
similarities among $\{ S_{[\redez]}(\al) \}$ and $\{ S_{[\redez']}(\al') \}$ for $[\redez]$ and $[\redez']$ in the same $r$-cluster point $\lf \redez \rf$ (Corollary \ref{cor: similarity}).

\medskip

\noindent {\bf Acknowledgements.} The first author would like to
express his sincere gratitude to Euiyong Park for many fruitful
discussions.

\section{Auslander-Reiten quivers arising from Dynkin quivers and orders on the set of positive roots }\label{Sec:Aus}

In this section, we recall combinatorial constructions of
Auslander-Reiten quivers from Dynkin quivers . We refer to
\cite{ARS,ASS,Gab80,Lus90} for the basic theories on quiver
representations and Auslander-Reiten quivers. For the combinatorial
properties, we refer to \cite{B99, Oh14A}.

\subsection{Auslander-Reiten quivers} \label{Subsec:Gab-Aus}
Let $\cmA=(a_{ij})_{i, j\in I}$ for $I=\{1, \cdots, n\}$ be a Cartan
matrix of  a finite-dimensional simple Lie algebra $\g$. Let $\Delta$ be the Dynkin diagram associated to $\cmA$.

\begin{definition}
For vertices $i,j\in I$ in $\Delta$, the number of edges between $i$ and $j$ is called the {\it distance} between $i$ and $j$ is denoted by $d_{\Delta}(i,j).$
\end{definition}

We denote by $\Pi_n=\{\alpha_i \, | \, i\in I\}$ the set of simple roots,
$\Phi$ the set of roots, $\Phi^+$ (resp$.\Phi^-$) the set of
positive roots (resp. negative roots). Let  $\{ \epsilon_i \ | \ 1 \le i \le m \}$ be the set of orthonormal basis of $\mathbb{C}^m$.
The free abelian group $\rl\seteq\oplus_{i \in I} \Z \alpha_i$ is called the \emph{root
lattice}. Set $\rl^{+}= \sum_{i \in I} \Z_{\ge 0}
\alpha_i\subset\rl$ and $\rl^{-}= \sum_{i \in I} \Z_{\le0}
\alpha_i\subset\rl$. For $\beta=\sum_{i\in I}m_i\alpha_i\in\rl^+$,
we set $\het(\beta)=\sum_{i\in I}m_i$. Let $(\cdot, \cdot)$ be the
the symmetric bilinear form on $\rl \times \rl$ (we refer \cite[Plate I$\sim$IX]{Bour}).

A Dynkin quiver $Q$ is obtained by adding an orientation to each
edge in the Dynkin diagram $\Delta$.
In other words,  $Q=(Q^0, Q^1)$ where $Q^0$ is the set of vertices
indexed by $I$ and $Q^1$  is the set of oriented edges with the
underlying graph $\Delta$. We say the vertex $i$ is a sink (resp.
source) if every edge between $i$ and $j$ is oriented as follows: $j
\to i$ (resp. $i\to j$).

Let Mod$(\C Q)$ be the category of finite dimensional modules over
the path algebra $\C Q$. An object $M\in \Mod \,\C Q$ consists of
the following data:
\begin{enumerate}
\item a finite dimensional module $M_i$ for each  $i\in Q^0$, 
\item a linear map $\psi_{i\to j}: M_i\to M_j$ for each oriented edge $i\to j$.
\end{enumerate}

The {\it dimension vector} of the module $M$ is $\Dim \, M =
\sum_{i\in I} (\dim \, M_i) \alpha_i$ and a simple object in $\Mod
\, \C Q$ is $S(i)$ for some $i\in I$ where $\Dim S(i)= \alpha_i$. In
$\Mod \, \C Q$, the set of isomorphism classes $[M]$ of
indecomposable modules is denoted by $\Ind \, Q$.

\begin{theorem}[Gabriel's theorem] \label{Thm:Gabr}
Let $Q$ and $\Phi^+$ be a Dynkin quiver and the set of positive
roots of type $A_n$, $D_n$ or $E_n$. Then there is a bijection
between $\Ind \, Q$ and $\Phi^+:$
\[ [M] \mapsto \Dim M.\]
\end{theorem}

\vskip 3mm

The Weyl group $\W$ of a finite type with rank $n$ is generated by
simple reflections  $s_i \in \Aut(\rl),$ $i\in I$.
Note that $$\big( w (\al), w (\be) \big)=(\alpha, \beta)$$ for any $w
\in W$ and $\al,\be \in \rl$.

For $w\in \W$, the length of $w$ is
$$\ell(w)=\min\{l\in \Z_{\geq 0} \, | \, s_{i_1}\cdots s_{i_l}= w, \ s_{i_k} \text{ are simple reflections }\}.$$
If $w= s_{i_1} s_{i_2} \cdots  s_{i_{\ell(w)}}$ then the sequence of
simple reflections $\redex=(s_{i_1},  \cdots, s_{i_{\ell(w)}})$ is
called a reduced expression associated to $w$.

We denote  by $w_0$ the longest element in $\W$ and by $^*$ the
involution on $I$ induced by $w_0$; i.e,
\begin{align} \label{eq: star involution}
w_0(\al_i) \seteq -\al_{i^*} \quad \text{ for all } i \in I.
\end{align}

For $w\in \W$, consider the subset (\cite{Bour})
\begin{align} \label{eq: cahracterization}
\Phi(w)=\{\alpha \in \Phi^+ \, | \, w^{-1}(\alpha)\in \Phi^-\}= \{ s_{i_1}s_{i_2} \cdots s_{i_{k-1}}(\alpha_{i_k}) \, | \, k=1, \cdots l\} \text{ such that } |\Phi(w)|=\ell(w).
\end{align}


Let us  consider the action of a simple reflection $s_i$, $i\in I$,
on the set
\[ \lf Q \rf =\{ Q\, |\, Q \text{ is a Dynkin quiver with $\Delta$ as the underlying diagram}\},\]
defined by $s_i (Q)=Q'$, where $s_i(j)=j$ for all $j\in Q_0$ and
\begin{equation*}
(j\to k) \mapsto  \left\{
\begin{array}{ll}
(j\leftarrow k) & \text{ if $j=i$ and $i$ is a source in $Q$,} \\
(j\leftarrow k) & \text{ if $k=i$ and $i$ is a sink in $Q$,} \\
 (j \to k) & \text{ otherwise, }
\end{array} \right. \quad \text{for all $j\to k \in Q_1$.}
\end{equation*}

\begin{definition} \label{def: adapted}
A reduced expression $\redex=(s_{i_1},\cdots ,s_{i_l})$ of $w$ is
said to be {\it adapted} to a Dynkin quiver $Q$ if
\[  i_k \text{ is a sink of } Q_{k-1}= s_{i_{k-1}}\cdots s_{i_1}(Q).\]
\end{definition}

The followings are well-known:
\begin{eqnarray} &&
  \parbox{95ex}{
\begin{enumerate}
\item[{\rm (i)}] A reduced expression $\redex_0$ of $w_0$ is adapted to at most one Dynkin quiver $Q$.
\item[{\rm (ii)}] For each Dynkin quiver $Q$, there is a reduced expression $\redez$ of $w_0$ adapted to $Q$.
\end{enumerate}
}\label{eq: well-known adapted}
\end{eqnarray}

Note that the converse of \eqref{eq: well-known adapted} {\rm (i)}
is not true; that is, two different reduced expressions of $w_0$ can
be adapted to the same Dynkin quiver $Q$. Actually, we can assign a
{\it class} of reduced expressions of $w_0$ to each Dynkin quiver
$Q$.

\begin{definition} \cite{B99,Lus90} \label{Def:comm_equi}
Let $\redex=(s_{i_1},s_{i_2},\cdots ,s_{i_k})$ and $\redex'=(s_{i_1'},s_{i_2'},\cdots, s_{i_k'})$ be
reduced expressions of $w\in W$. If $\redex'$ can be obtained by a
sequence of commutation relations, $s_is_j=s_js_i$ for $d_{\Delta}(i,j)>1$, from $\redex$ then we say $\redex$
and $\redex'$ are {\it commutation equivalent} and write $\redex
\sim \redex'$. The {\it equivalence class} of $\redex$ is denoted by
$[\redex]$.
\end{definition}

\begin{proposition} \cite{B99,Lus90} \label{Prop:Q-W}
Reduced expressions $\redez=(s_{i_1}, s_{i_2},\cdots, s_{i_l})$ and
$\redez'=(s_{i'_1}, s_{i'_2},\cdots, s_{i'_l})$ of $w_0$ are adapted
to a quiver $Q$ if and only if  $\redez \sim\redez'$ and $\redez$ is
adapted to $Q$.
\end{proposition}
Thus we can denote by $[Q]$ the equivalence class of $w_0$
consisting of all reduced expressions adapted to the Dynkin quiver
$Q$.

\medskip

For each Dynkin quiver $Q$, there is a unique {\it Coxeter element}
$\phi_Q$ in $\W$ such that all of reduced expressions of $\phi_Q$ are adapted to $Q$: The element $\phi_Q\in \W$ has
the form
$$ \phi_Q= s_{i_1} s_{i_{2}}\cdots s_{i_n} \qquad
\text{where $\{ i_1,\ldots ,i_n \}=I$.}$$

Now we recall the Auslander-Reiten (AR) quiver $\Gamma_Q$ associated
to a Dynkin quiver $Q$ of type $A_n$, $D_n$, or $E_n$. {\it For the
rest of this section, we assume that $Q$ is a Dynkin quiver of type
$A_n$, $D_n$, or $E_n$.} Let us denote by $\text{Ind}\, Q$ the
set of isomorphism classes $[M]$ of indecomposable modules in
$\text{Mod}\, \C Q$, where $\text{Mod}\, \C Q$ is the category of
finite dimensional modules over the path algebra $\C Q.$

\begin{definition} \label{Def:AR}
Let $\redez$ be a reduced expression of $w_0$ adapted to a Dynkin
quiver $Q$. The quiver $\Gamma_Q=(\Gamma_Q^0, \Gamma_Q^1)$ is called
the Auslander-Reiten quiver (AR quiver)  if
\begin{enumerate}
\item
each vertex $V_M$ in $\Gamma_Q^0$ corresponds to an isomorphism
class $[M]$ in $\text{Ind}\, Q$,
\item an arrow $V_M\to V_{M'}$ in $\Gamma_Q^1$ implies that there exists an irreducible morphism $M\to M'$.
\end{enumerate}
\end{definition}

Gabriel's theorem (Theorem \ref{Thm:Gabr}) tells that there is a
natural one-to-one correspondence between the set $\Gamma_Q^0$ of
vertices in $\Gamma_Q$ and the set $\Phi^+$ of positive roots.
Hence we use $\Phi^+$ as the index set $\Gamma_Q^0$.

The quiver $\Gamma_Q$ of type $A_n$, $D_n$ and $E_n$  can be
obtained by a purely combinatorial method. In order to show this,
we introduce another quiver $A(Q)$ below. 
Denote by $\h$ the Coxeter number and $\mathsf{a}_i$ (resp.
$\mathsf{b}_i$) the number of arrows in $Q$ directed to the vertex
$i$  (resp. $i^*$) between the vertices indexed by $i$ and $i^*$
(see \cite{B99,Gab80,R80}).
\begin{enumerate}
\item Consider the quiver $\N Q$ whose vertices are indexed by $I\times \N$ and the set of arrows is $\{(i,m) \to (j,m),\, (j, m)\to(i, m-1)\, |\, i\to j  \in Q_1 \}$.
\item The subquiver $A(Q)$ of $\N Q$ consists of the vertices $\{(i,m) \, | \, 1\leq m \leq r_i\},$ where
\begin{equation} \label{Eqn:vertex_numbers}
 \mathsf{r}_i= (\h+\mathsf{a}_i-\mathsf{b}_i)/2.
\end{equation}
\end{enumerate}

The following proposition shows the relation between two quivers
$A(Q)$ and $\Gamma_Q$.

\begin{proposition} \cite{B99,R80} \label{Prop:IsomBtwQuivers}
As quivers, $A(Q)$ is isomorphic to $\Gamma_Q$. More precisely, let $\redez=(s_{i_1}, s_{i_2},\cdots, s_{i_l}) \in [Q]$. The isomorphism $\iota_Q:A(Q)\to \Gamma_Q$ is given by the following correspondence between their vertices$:$
\begin{equation} \label{Eqn: AR_vertex}
(i,m) \longleftrightarrow \beta = s_{i_1}s_{i_1} \cdots  s_{i_{k-1}},
\alpha_{i_k} \in \Phi^+  \ \text{ for } m= \#\{i_t=i\, | \, 1 \le i \le k-1\}+1 \text{ and } i=i_k.
\end{equation}
Here the value $(i,m)$ corresponding to $\beta$ does {\it not} depend on the choice of reduced
expression $\redez$ of $w_0$.
\end{proposition}

We call the $i$ of $\beta$ in \eqref{Eqn: AR_vertex} the {\it
residue} of $\beta$ (with respect to $Q$). By the above proposition,
the quiver $\Gamma_Q$ does {\it not} depend on the choice of reduced
expression $\redez$ in $[Q]$.

\begin{remark}
We sometimes denote by
$\Gamma_{[\redez]}$ for $\redez \in [Q]$ instead of $\Gamma_Q$ to
emphasize that it {\it does} depend {\it only} on the equivalent
class $[\redez]$.
\end{remark}

\begin{example} \label{Ex:adapted AR quiver Aprime}
{\rm 1)} Let $\redez=(s_1,s_3, s_2, s_4, s_1, s_3, s_5, s_2, s_4,
s_1, s_3, s_5, s_2, s_4, s_1)$ of $A_5$, which is  adapted to the
Dynkin quiver $ Q ={\xymatrix@R=3ex{ *{ \circ
}<3pt> \ar@{<-}[r]_<{  1}  &*{\circ}<3pt> \ar@{->}[r]_<{  2}
&*{\circ}<3pt> \ar@{<-}[r]_<{  3} &*{\circ}<3pt>
\ar@{<-}[r]_<{  4} & *{\circ}<3pt> \ar@{-}[l]^<{\ \
5} }}. $

{\rm 2)} The quiver $\N Q$ is
$$
 \scalebox{0.84}
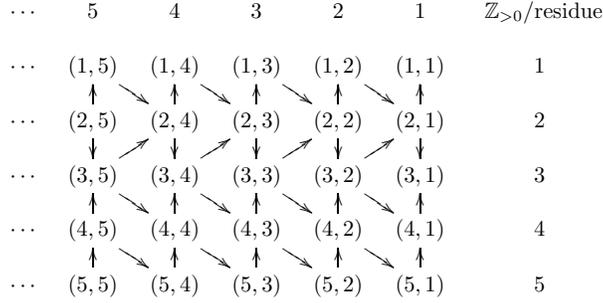
{\xymatrix@C=2ex@R=2ex{
\cdots &5&4&3&2&1& \text{ $\Z_{>0}$/residue} \\
 \cdots & (1,5)\ar@{->}[dr]    & (1,4)\ar@{->}[dr]  & (1,3)\ar@{->}[dr] &(1,2) \ar@{->}[dr] & (1,1) & 1  \\
\cdots &  (2,5) \ar@{->}[u]\ar@{->}[d] &(2,4) \ar@{->}[u]\ar@{->}[d]&(2,3) \ar@{->}[u]\ar@{->}[d]&(2,2) \ar@{->}[u]\ar@{->}[d]& (2,1) \ar@{->}[u]\ar@{->}[d] & 2\\
\cdots & (3,5) \ar@{->}[ur]\ar@{->}[dr] & (3,4)\ar@{->}[ur]\ar@{->}[dr] & (3,3)\ar@{->}[ur]\ar@{->}[dr] & (3,2)\ar@{->}[ur]\ar@{->}[dr] & (3,1) & 3\\
\cdots & (4,5) \ar@{->}[u]\ar@{->}[dr] &  (4,4) \ar@{->}[u]\ar@{->}[dr] &  (4,3) \ar@{->}[u]\ar@{->}[dr]  & (4,2) \ar@{->}[u] \ar@{->}[dr] & (4,1) \ar@{->}[u] & 4\\
\cdots & (5,5) \ar@{->}[u] &(5,4) \ar@{->}[u] &(5,3) \ar@{->}[u]
&(5,2) \ar@{->}[u] &(5,1) \ar@{->}[u] & 5 }}
$$
{\rm 3)} Compute
$\mathsf{r}_i=(\mathsf{h}+\mathsf{a}_i-\mathsf{b}_i)/2$. In this
case $\big(\mathsf{r}_i \, | \, i \in I \big)=(4,3,3,3,2)$ since
$\mathsf{h}=6$. Take the finite subquiver $A(Q)$ of $\N Q$
consisting of the vertices $\{(i,j)| 1\leq j\leq \mathsf{r}_i\}$.
Then $A(Q)$ is isomorphic to $\Gamma_Q$ by the map $\iota_Q:A(Q)\to
\Gamma_Q$:
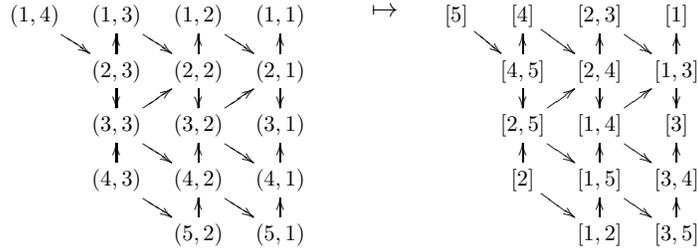
$$  \scalebox{0.84}{\xymatrix@C=2ex@R=2ex{
   & (1,4)\ar@{->}[dr]  & (1,3)\ar@{->}[dr] &(1,2) \ar@{->}[dr] & (1,1)  \\
 &&(2,3) \ar@{->}[u]\ar@{->}[d]&(2,2) \ar@{->}[u]\ar@{->}[d]& (2,1) \ar@{->}[u]\ar@{->}[d]\\
 & & (3,3)\ar@{->}[ur]\ar@{->}[dr] & (3,2)\ar@{->}[ur]\ar@{->}[dr] & (3,1) \\
 & &  (4,3) \ar@{->}[u]\ar@{->}[dr]  & (4,2) \ar@{->}[u] \ar@{->}[dr] & (4,1) \ar@{->}[u]\\
 && &(5,2) \ar@{->}[u] &(5,1) \ar@{->}[u] }}  \qquad \mapsto
   \scalebox{0.84}{\xymatrix@C=2ex@R=2ex{
   & [5]\ar@{->}[dr]  & [4]\ar@{->}[dr] &[2,3] \ar@{->}[dr] & [1] \\
 &&[4,5] \ar@{->}[u]\ar@{->}[d]&[2,4] \ar@{->}[u]\ar@{->}[d]&[1,3] \ar@{->}[u]\ar@{->}[d]\\
 & & [2,5]\ar@{->}[ur]\ar@{->}[dr] & [1,4]\ar@{->}[ur]\ar@{->}[dr] & [3] \\
 & &  [2] \ar@{->}[u]\ar@{->}[dr]  & [1,5] \ar@{->}[u] \ar@{->}[dr] & [3,4] \ar@{->}[u]\\
 && &[1,2] \ar@{->}[u] &[3,5] \ar@{->}[u] }}
$$
where $[i,j]=\sum_{k=i}^j \alpha_k \in \PR$ and
$[i,i]=[i]=\alpha_i$.
\end{example}

\vskip 2mm


Now we introduce the height function $\xi:I\to \Z$ associated to $Q$
and describe another combinatorial description of $\Gamma_Q$ using
the height function and the Coxeter element $\phi_Q$ (see
\cite{HL11}).

\begin{definition} \label{Def:height}
The {\it height function}  $\xi$  associated to the quiver $Q$ is a
map on $Q$ satisfying $\xi(j)=\xi(i)-1 \in \Z$ if there exists an
arrow $i\to j$.
\end{definition}


\medskip

The {\it repetition quiver} $\Z Q$ of $Q$ has the set of vertices
$$  (\Z Q)^0 =\{(i,p)\in I\times\Z \ | \ p-\xi(i)\in 2\Z\}$$
and arrows $ (\Z Q)^1=\{(j,p+1)\to (i, p), \, (i,p)\to (j, p-1)\, |
\,  \text{ $i$ and $j$ as connected in $Q$}\}$. For $i \in I$, we
define positive roots $\gamma_i$ and $\theta_i$ in the following
way:
\begin{align} \label{eq: gam thet}
\gamma_i = \sum_{j \in \overset{\gets}{i}} \alpha_j \quad
\text{and} \quad\theta_i = \sum_{j \in \overset{\to}{i}} \alpha_j
\qquad \text{ where }
\end{align}
\begin{itemize}
\item $\overset{\gets}{i}$ is the set of vertices $j$  in  $Q^0$ 
such that there exists a path from $i$ to $j$,
\item $\overset{\to}{i}$ is the set of vertices $j$ in  $Q^0$ 
such that there exists a path from $j$ to $i$.
\end{itemize}
%
%
Note that $\{ \gamma_i \, | \, i \in I \} = \Phi(\phi_Q)$ and $\{
\theta_i \, | \, i \in I \} = \Phi(\phi^{-1}_Q)$.

Consider the map $\pi:\Phi^+ \to (\Z Q)_0$ such that
 \begin{equation} \label{Eqn:label_B}
  \gamma_i \mapsto(i, \xi(i)), \qquad \phi_Q(\alpha)\mapsto (i, p-2)\quad \text{ if } \pi(\alpha)= (i,p) \text{ and } \phi_Q(\alpha), \, \alpha \in\Phi^+.
  \end{equation}
Then we have the following:
\begin{enumerate}
\item The subquiver $B(Q)$ of $\Z Q$ consisting of $\pi(\Phi^+)$ is the same as the quiver $\Gamma_Q$ by identifying their vertices as $\Phi^+$.
\item Recall that $A(Q)$ in Proposition \ref{Prop:IsomBtwQuivers} is isomorphic to $\Gamma_Q$.
The map $A(Q)\to B(Q)$ is given by $(i,m)\mapsto(i, \xi(i)-2(m-1))$
by considering coordinates of all $\beta \in \Phi^+$.
\end{enumerate}

Since $A(Q)$ and $B(Q)$ are isomorphic quivers to $\Gamma_Q$,
indices of $A(Q)$ and $B(Q)$ give coordinates to positive roots in
$\Gamma_Q$. The coordinate induced by $B(Q)$ has meanings in the
description of reflection map related to $[\redez]$ for $\redez$
which is adapted to some Dynkin quiver $Q$ (see \eqref{ref Q h}
below).

\begin{definition} \label{def: sectional path}
A path $\beta_0 \to \beta_1 \to \cdots \beta_s$ in $\Gamma_Q$ is called a {\it sectional} path if
$\phi_Q(\beta_{k+1})=\beta_{k-1}$ for all $1 \le k \le s-1$.
\end{definition}

\begin{example}
The AR quiver $\Gamma_Q$ associated to $Q$  in Example
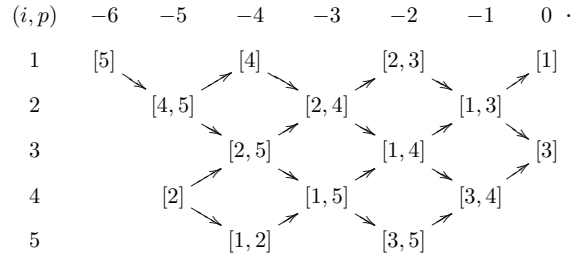
\ref{Ex:adapted AR quiver Aprime} is
\[ \scalebox{0.84}{\ \xymatrix@C=2ex@R=1ex{
( i,p ) &-6&-5&-4&-3&-2&-1&0\\
1& [5]\ar@{->}[dr] & & [4]\ar@{->}[dr] & &[2,3] \ar@{->}[dr] & &[1]  \\
2&&[4,5]\ar@{->}[ur]\ar@{->}[dr]& &[2,4] \ar@{->}[ur]\ar@{->}[dr]&& [1,3] \ar@{->}[ur]\ar@{->}[dr]\\
3&&& [2,5]\ar@{->}[ur]\ar@{->}[dr] && [1,4]\ar@{->}[ur]\ar@{->}[dr] && [3]\\
4&&  [2] \ar@{->}[ur]\ar@{->}[dr]  && [1,5] \ar@{->}[ur] \ar@{->}[dr] && [3,4] \ar@{->}[ur]\\
5&&&[1,2] \ar@{->}[ur] &&[3,5] \ar@{->}[ur]}}.
\]
Here the coordinate $(i,p)$ is induced from that of $B(Q)$.
\end{example}

Combinatorially, a path is sectional if the path is {\it upwards} or {\it downwards} in $\Gamma_Q$ under our convention.

\vskip 3mm

\subsection{Partial orders on $\Phi(w)$}
Let $w$ be an element in $\W$ of finite type. An order $\preceq$ on
the set $\Phi(w)$ is said to be {\it convex} if \[\text{ $\alpha,
\beta, \alpha+\beta \in \Phi(w)$ and $\alpha\preceq \beta$ implies
$\alpha\preceq \alpha+\beta \preceq \beta$. }\]

Each reduced expression $\redex$ of $w\in \W$ induces a total order
on $\Phi(w)$ using the position function defined as follows:

\begin{definition} \label{Def:pos,F}
Let $\redex= (s_{i_1},s_{i_2}, \cdots, s_{i_l})$ be a reduced
expression of $w$.
 The {\it position function } $\pi_{\redex}: \Phi(w) \to \N$ associated to the reduced expression $\redex$ is
 defined by $\beta^\redex_k=s_{i_1}s_{i_2} \cdots s_{i_{k-1}}(\alpha_{i_k})\mapsto k$.
\end{definition}

\begin{definition}
The total order $<_\redex$ on $\Phi(w)$ associated to $\redex$ is
defined by
$$ \beta^\redex_j <_\redex \beta^\redex_k \quad \text{ if and only if }  \quad  j=\pi_\redex(\beta^\redex_j)<\pi_\redex(\beta^\redex_k)=k.$$
\end{definition}


\begin{definition} 
Let $\alpha, \beta\in \Phi(w)\subset \Phi^+$. We define an order
$\prec_{[\redex]}$ on $\Phi(w)$ as follows:
$$\alpha\prec_{[\redex]}\beta   \quad \text{ if and only if }  \quad   \alpha <_{\redex'}\beta \quad \text{ for any } \ \redex' \in [\redex].$$
\end{definition}

\begin{proposition} \cite{Papi94}
The total order $<_{\redex}$  and the partial order
$\prec_{[\redex]}$ are convex orders on $\Phi(w)$.
\end{proposition}

The AR quiver $\Gamma_Q$ visualizes the convex partial order
$\prec_{[Q]}$ when $Q$ is a Dynkin quiver $Q$ of type ADE in the
following sense:

\begin{proposition}\cite{R96} \label{Prop:par_ord/ar}
For $\redez \in [Q]$ and $\alpha, \beta\in \Phi^+$, we have $\alpha\prec_{[\redez]} \beta$  if and only if there is a path from
$\beta$ to $\alpha$ in $\Gamma_{[\redez]}$.
\end{proposition}

Recall that the order of $\Phi(w)$ via $<_\redex$ is determined by
the value of the position function $\pi_\redex$. Now, we introduce a
function called level function $\lambda_{\redex}$, which is closely
related to the partial order $\prec_{[\redex]}$.

\begin{definition} \cite{B99} \label{Def:lev,F}
Let $\redex= (s_{i_1},s_{i_2}, \cdots, s_{i_l})$ be a reduced
expression of $w$.
 Given $\alpha\in \Phi(w)$, let
\begin{equation} \label{Eqn:seq}
\beta_1, \beta_2, \cdots, \beta_k=\alpha
\end{equation}
be a sequence of distinct elements of $\Phi(w)$ ending with $\alpha$
such that $\pi_\redex(\beta_{i-1})<\pi_\redex(\beta_i)$ and
$(\beta_i, \beta_{i-1})\neq 0$. The function
$\lambda_{\redex}:\Phi(w)\to \N$ associated to the reduced
expression $\redex$ is defined as follows:
 \[\lambda_{\redex} (\alpha)=\max \left\{k\geq 1 \, | \, \beta_1, \beta_2, \cdots ,\beta_k=\alpha \text{ is the sequence in }(\ref{Eqn:seq})  \right\} .\]
We call it {\it the level function} associated to $\redex$.
\end{definition}

\begin{proposition} \cite{B99}
Two reduced expressions $\redex$ and $\redex'$ of $w$ are
commutation equivalent if and only if
$\lambda_{\redex}=\lambda_{\redex'}$.
\end{proposition}

The above proposition tells that the level function
$\lambda_{\redex}$ {\it does} depend only on the equivalent class of
$\redex$ and hence we shall write $\lambda_{[\redex]}$ instead of
$\lambda_\redex$. In particular, the level function
$\lambda_{[\redez]}$ for $\redez \in [Q]$ is closely related to the
AR quiver $\Gamma_Q$.

\begin{proposition}\cite{B99} \label{Prop:levelF_adaptedAR}
For $Q$ and $\redez \in [Q]$, define a function
$\lambda_{\Gamma_Q}:\Phi^+\to \N$ in an inductive way:
$$ \alpha \mapsto \left\{
\begin{array}{ll}
1, & \text{ if $\alpha$ is a sink in $\Gamma_Q$}, \\
\max\{\lambda_{\Gamma_Q}(\beta) \, | \, \alpha\to \beta \text{ in }
\Gamma_Q\}+1, & \text{ otherwise. }
\end{array}
 \right.$$
Then we have $\lambda_{\Gamma_Q} =\lambda_{[\redez]}$.
\end{proposition}




Another closely related notion to the AR quiver $\Gamma_Q$ is compatible readings  of positive roots. To see the relation, we introduce a {\it compatible
reading} of $\Gamma_Q.$
A sequence $s_{i_1}, \cdots, s_{i_N}$ (resp. $i_1, \cdots, i_N$) of
simple reflections  (resp. indices) is called a {\it compatible
reading} of the AR quiver $\Gamma_Q$ if
\[\text{   whenever there is an arrow from $(i_q, n_q)$ to $(i_r, n_r)$ in $A(Q)\simeq\Gamma_{Q}$, read $s_{i_r}$ before $s_{i_q}$}.\]
According to Proposition \ref{Prop:par_ord/ar}, a compatible reading
of $\Gamma_Q$ gives a compatible reading of positive roots, in the
sense that $\alpha$ is read before $\beta$ if
$\alpha\prec_{[Q]}\beta$ for $\alpha, \beta\in \Phi^+$.

\begin{theorem} \cite{B99} \label{Thm:adapted_compatible reading}
Let $Q$ be a Dynkin quiver of finite type $A_n$, $D_n$, $E_n$. Then
any reduced expression of $w_0\in \W$ adapted to the quiver $Q$ can
be obtained by a compatible reading of the AR quiver $\Gamma_{Q}$.
\end{theorem}

\vskip 3mm

\section{Combinatorial AR-quivers and related convex partial orders} \label{Sec: Upsilon}

In this section, we shall construct a quiver which visualizes the
convex partial order $\prec_{[\redex]}$ for {\it any} $w$ of {\it
all} finite Weyl group $\W$ and its reduced expression $\redex$.

\subsection{Combinatorial AR-quivers}

Let us take  \[\redex=(s_{i_1}, s_{i_2}, s_{i_3},\cdots, s_{i_{\ell(w)}})\]
a reduced expression of an element $w\in \W$. Then we can label
$\Phi(w)$ as follows:
$$  \beta^\redex_k = s_{i_1}\cdots s_{i_{k-1}}(\al_{k}) \quad \text{ for } 1 \le k \le \ell(w).$$
Recall that the {\it residue} of $\beta^\redex_k$ is defined by $\text{res}(\beta^\redex_k)=i_k.$

\begin{algorithm} \label{Alg_AbsAR} 
The quiver $\Upsilon_{\redex}=(\Upsilon^0_{\redex},
\Upsilon^1_{\redex})$ associated to $\redex$ is constructed in the
following algorithm:
\begin{enumerate}
\item[{\rm (Q1)}] $\Upsilon_{\redex}^0$ consists of $\ell(w)$ vertices labeled by $\beta^\redex_1, \cdots, \beta^\redex_{\ell(w)}$.
\item[{\rm (Q2)}] There is an arrow from $\beta^\redex_k$ to $\beta^\redex_j$ if
$$ {\rm (i)} \  k>j, \ \ {\rm (ii)} \  d_{\Delta}(i_k,i_j)=1 \ \text{ and }  \ {\rm (iii)} \ \{\, t\, | \,  j<t<k, \, i_t=i_j \text{ or } i_k \}=\emptyset.$$
\item[{\rm (Q3)}] Assign the color $m_{jk}=-(\alpha_{i_j}, \alpha_{i_k})$ to each arrow $\beta^\redex_k\to \beta^\redex_j$ in {\rm (Q2)}; that is,
$\beta^\redex_k \xrightarrow{m_{jk}} \beta^\redex_j$.  Replace
$\xrightarrow{1}$ by $\rightarrow$,  $\xrightarrow{2}$ by
$\Rightarrow$ and  $\xrightarrow{3}$ by  $\Rrightarrow$.
\end{enumerate}
\end{algorithm}

We call the quiver $\Upsilon_{\redex}$ the {\it combinatorial
AR-quiver associated to $\redex$}.

\begin{remark} \label{rmk: note} \hfill
\begin{enumerate}
\item  By replacing labels $\beta^\redex_k$'s  with $i_k$'s, one can obtain the usual Hasse quiver. To compute $\beta^\redex_k$, we need lots of computations in general. 
\item In Proposition \ref{the prop} and Proposition \ref{Prop:meaning of number in arrow}, we will show that if two roots $\alpha$ and $\beta$ are connected by a sequence of arrows with the {\it same direction},
then the product of colors of the arrows corresponds to $(\alpha,\beta)$ (see the propositions for rigorous statements).
\end{enumerate}
\end{remark}

The following proposition follows from the construction of the quiver $\Upsilon_{\redex}$: 

\begin{proposition} \label{Prop:AR_class}
If two reduced expressions $\redex$ and $\redex'$ are commutation
equivalent then $\Upsilon_{\redex}=\Upsilon_{\redex'}$.
\end{proposition}

By Proposition \ref{Prop:AR_class}, we can define $\Upsilon_{[\redex]}$ for an equivalence class $[\redex]$ of $w \in \W$.


\begin{example} \label{nonadapted AR}
Let  $\redex=(s_1,s_2,s_3,s_5,s_4,s_3,s_1,s_2,s_3,s_5,s_4,s_3,s_1)$
of $A_5$. Then one can easily check that $\redex$ is {\it not}
adapted to {\it any} Dynkin quiver $Q$ of type $A_5$. According to
Algorithm \ref{Alg_AbsAR}, labels of vertices of the combinatorial
AR quiver $\Upsilon_{[\redex]}$ are
\begin{align*}
&(\beta^\redex_k \, | \, 1 \le k \le \ell(w)=13) \\ & \hspace{7ex} =
([1],[1,2],[1,3],[5],[1,5],[4,5],[2],[2,5],[2,3],[1,4],[2,4],[4],[3,5]).
\end{align*}
The quiver $\Upsilon_{[\redex]}$ is drawn as follows:
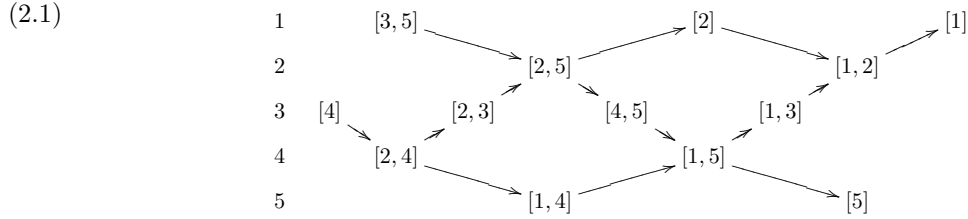
\begin{equation} \label{modi 1 complete prime}
 \scalebox{0.84}{\xymatrix@C=2ex@R=1ex{
1&& [3,5]\ar@{->}[drr] &&  && [2]\ar@{->}[drr] && && [1] \\
2&&  && [2,5] \ar@{->}[dr]\ar@{->}[urr] &&  && [1,2]\ar@{->}[urr]\\
3& [4] \ar@{->}[dr] && [2,3] \ar@{->}[ur] && [4,5]\ar@{->}[dr]&& [1,3]\ar@{->}[ur] \\
4&& [2,4]\ar@{->}[ur]\ar@{->}[drr]  && && [1,5]\ar@{->}[ur]\ar@{->}[drr] \\
5&&  && [1,4]\ar@{->}[urr] &&  && [5] }}
\end{equation}
\end{example}


\begin{example}
Let $\redez=(s_3,s_2, s_3, s_2, s_1, s_2, s_3, s_2, s_1)$ of $B_3$. 
The combinatorial AR quiver of $\redez$ is
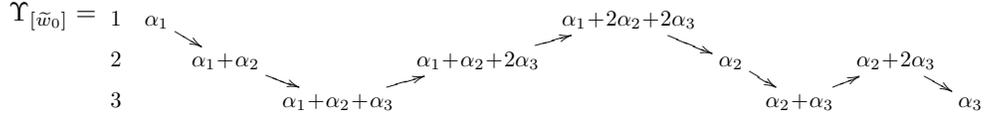
$$\Upsilon_{[\redez]}=
 \scalebox{0.84}{\xymatrix@C=1ex@R=1ex{
1& \alpha_1 \ar@{->}[dr] &&&&\alpha_1 \shp 2\alpha_2 \shp 2\alpha_3\ar@{->}[dr]&&&&\\
2&& \alpha_1 \shp \alpha_2 \ar@{->}[dr] && \alpha_1 \shp \alpha_2 \shp 2\alpha_3\ar@{->}[ur] && \alpha_2\ar@{->}[dr] && \alpha_2 \shp 2\alpha_3\ar@{->}[dr]& \\
3&&& \alpha_1 \shp \alpha_2 \shp \alpha_3\ar@{->}[ur] &&&&\alpha_2
\shp \alpha_3\ar@{->}[ur] && \alpha_3 }}
$$
\end{example}

\begin{example} A combinatorial AR quiver is not necessarily connected. For example, let $\redex=(s_4,s_3, s_1)$ of $A_4$. Then
\[ \Upsilon_{[\redex]}=
 \scalebox{0.84}{\xymatrix@C=1ex@R=1ex{
1& \alpha_1 &&\\
2&&&\\
3& &\alpha_3+\alpha_4 \ar@{->}[dr]& \\
4&&&\alpha_4 }}
\]

\end{example}

\begin{example} \label{D4 non-adapted D-1}
Let $\redez = (s_1,s_2,s_3,s_1,s_2,s_4,s_1,s_2,s_3,s_1,s_2,s_4)$ of
$D_4$. Note that $\redez$ is not adapted to any Dynkin quiver of
type $D_4$. We can draw the combinatorial AR quiver
$\Upsilon_{[\redez]}$ as follows:
$$\scalebox{0.84}
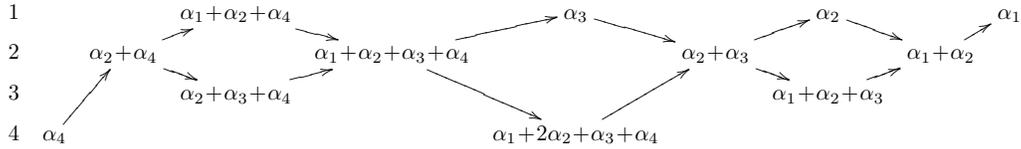
{\xymatrix@C=1ex@R=1ex{
1&&& \alpha_1 \shp \alpha_2 \shp \alpha_4\ar@{->}[dr] && \alpha_3 \ar@{->}[dr]&& \alpha_2 \ar@{->}[dr] && \alpha_1\\
2&& \alpha_2 \shp \alpha_4 \ar@{->}[ur]\ar@{->}[dr] && \alpha_1 \shp
\alpha_2 \shp \alpha_3 \shp \alpha_4 \ar@{->}[ur]\ar@{->}[ddr]
 && \alpha_2 \shp \alpha_3 \ar@{->}[ur]\ar@{->}[dr]&& \alpha_1 \shp \alpha_2 \ar@{->}[ur]\\
3& && \alpha_2 \shp \alpha_3 \shp \alpha_4 \ar@{->}[ur]&& && \alpha_1 \shp \alpha_2 \shp \alpha_3 \ar@{->}[ur] \\
4& \alpha_4 \ar@{->}[uur] &&  && \alpha_1 \shp 2\alpha_2 \shp
\alpha_3 \shp \alpha_4 \ar@{->}[uur]&& }}
$$
\end{example}

\begin{example}
Let $\redex=(s_1, s_2, s_1, s_2, s_1)$ of $G_2$. Then
$$ \Upsilon_{[\redex]}=
 \scalebox{0.84}{\xymatrix@C=2ex@R=1ex{
1& \alpha_1+3\alpha_2 \ar@3{->}[dr] && 2\alpha_1+3\alpha_2   \ar@3{->}[dr]  && \alpha_1 \\
2& &\alpha_1+2\alpha_2 \ar@3{->}[ur]  && \alpha_1+\alpha_2 \ar@3{->}[ur]  & \\
}}
$$
\end{example}

We showed that a combinatorial AR quiver is not necessarily
connected. The following proposition shows a property about
connectedness of a combinatorial AR quiver.

\begin{proposition} \label{Lem:sameresidue} For $w \in W$ and its reduced expression $\redex$ of $w$,
let $\beta^\redex_{k_1}$ and $\beta^\redex_{k_2}$ be two vertices in
$\Upsilon_{[\redex]}$. Suppose both of the two vertices have the same residue
$i$ and $k_1<k_2$. Then there is a path from
$\beta^\redex_{k_2}$ to $\beta^\redex_{k_1}$ in
$\Upsilon_{[\redex]}$.
\end{proposition}

\begin{proof}

It is enough to show when there is no vertex $\beta_{k}^\redex$ with the residue $i$ such that $k_1<k<k_2$.

Suppose there is no index $i'\in I$ which satisfies the following property:
\begin{equation} \label{eq: condition}
\text{there is a vertex $\beta^\redex_{k_3}\in \Upsilon_{[\redex]}^0$ with the residue $i'$ for some $k_1<k_3<k_2$ and $d_\Delta(i,i')=1$.}
\end{equation}
Then it is a contradiction to the fact that $\redex$ is a reduced
expression. Hence there is $i'\in I$ satisfying \eqref{eq: condition}. \vskip 1mm
Take $i'\in I$ which satisfies (1) and (2) above.
\begin{enumerate}[(i)]
\item If there is a unique vertex $\beta^\redex_{k_3}\in \Upsilon_{[\redex]}^0$ with the residue  $i'$-th such that $k_1<k_3< k_2$ then there is a path from
$\beta^\redex_{k_2}$ to  $\beta^\redex_{k_1}$ via
$\beta^\redex_{k_3}$.
\item  If there are more than one vertices $k_3, k_4, \cdots$ with the residue $i'$ such that $k_1<k_3, k_4, \cdots <k_2$. Without loss of generality,
let us assume there are arrows from $\beta^\redex_{k_2}$ to
$\beta^\redex_{k_4}$ and from $\beta^\redex_{k_3}$ to
$\beta^\redex_{k_1}$ in $\Upsilon_{[\redex]}$. Then it is enough
show there is a path from $\beta^\redex_{k_4}$ to
$\beta^\redex_{k_3}$ in $\Upsilon_{[\redex]}$. Again, we can assume
that  there is no vertex $\beta^\redex_{l}$ with the residue $i'$
such that $k_3<l<k_4$. Inductively, we can reduce the situation to
the case (i).
\end{enumerate}
Hence we proved that there is a path from $\beta^\redex_{k_2}$ to
$\beta^\redex_{k_1}$.
\end{proof}

\begin{proposition}
Let $\redex$ be a reduced expression consisting of simple
reflections $\{s_{i_1}, \cdots, s_{i_k}\}$. The subdiagram of
$\Delta$ consisting of the set of indices $\{i_1, \cdots, i_k\}$  is
connected if and only if $\Upsilon_{[\redex]}$ is connected.
\end{proposition}

Recall that the level function $\lambda_{[\redex]}$ is defined for
any reduced expression $\redex\in W$ of any finite type. In the
adapted cases (Definition \ref{Def:lev,F} and Proposition
\ref{Prop:levelF_adaptedAR}), the level function $\lambda_{[Q]}$ is
visualized by $\Gamma_{Q}$. More precisely,
\begin{itemize}
\item the existence of an arrow between $\al$ and $\beta$ in $\Gamma_Q$ implies $(\al,\be) \ne 0$,
\item $\lambda_{[Q]}(\al)$ is the length of the longest path in $\Gamma_Q$ starting from $\al$.
\end{itemize}
Now we shall prove that $\Upsilon_{[\redex]}$ plays the roles of
$\Gamma_Q$ for {\it any} $\redex$ of $w$ in {\it any} finite Weyl
group of $W$.

\vskip 2mm

Now we introduce paths in combinatorial AR quivers which have a particular property.

\begin{definition}
For a combinatorial AR quiver $\Upsilon_{[\redex]}$, let us choose positive roots $\alpha, \beta \in \Phi(w)$ which are connected by a path.  The smallest number of arrows consisting a path between $\alpha$ and $\beta$ is called the {\it distance} between $\alpha$ and $\beta$ in $\Upsilon_{[\redex]}$ and denote it by $d_{[\redex]}(\alpha,\beta).$
\end{definition}

\begin{definition} \label{Def:sectional}
Consider a path $P$ in a combinatorial AR quiver $\Upsilon_{[\redex]}$. If a pair of roots $\alpha, \beta$ in $P$ whose residues are $i$ and $j$ satisfies
\[ d_{[\redex]}(\alpha,\beta)= d_{\Delta}(i,j)\]
then the path is called a {\it sectional path}.
\end{definition}

When $[\redex]=[Q]$ for some Dynkin quiver $Q$ of type ADE, the above definition coincides with Definition \ref{def: sectional path}.

\begin{example}
In Example \ref{nonadapted AR}, $[2,4]$ and $[2]$  whose residues are $4$ and $1$ lie in the sectional path:
$$ [2,4]_4 \to [2,4]_3 \to [2,5]_2 \to [2]_1$$
Here each subindex $_{i}$ denotes the residue for its vertex.
\end{example}

\begin{remark} \cite[Section 3]{S14}
In the theory of AR quivers for the path algebra $\C Q$, sectional paths provide information on $\dim(M,N)$ and ${\rm Ext}^1(M,N)$ for $M,N \in \Ind \, Q$.
\end{remark}

\begin{proposition} \label{the prop}
Let $\alpha$ and $\beta$ have residues $i$ and $j$ in the combinatorial Auslander-Reiten quiver $\Upsilon_{[\redex]}$. If $d_{[\redex]}(\alpha,\beta)=1$ then
 we have $(\alpha, \beta)=-(\alpha_{i}, \alpha_{j})> 0$.
\end{proposition}

\begin{proof}
Take a reduced expression $\redex=(s_{i_1}, \cdots, s_{i_{\ell(w)}})
\in [\redex]$ and  denote $\alpha=\beta^\redex_k$ and
$\beta=\beta^\redex_l$ for $1\leq k<l\leq {\ell(w)}$. Then the arrow
is directed from $\beta$ to $\alpha$. If $l=k+1$, then our assertion
follows from the formula below:
$$(\alpha, \beta)=(s_{i_1} \cdots s_{i_{k-1}}(\alpha_{i_k}), s_{i_1} \cdots s_{i_{k}}(\alpha_{i_l}))=(-\alpha_{i_k}, \alpha_{i_l}).$$

Assume that $l > k+1$ and set $\redex_{k \le \cdot \le l} \seteq
(s_{i_k}, \ldots, s_{i_l})$.  It is enough to show that there exists
a reduced expression $\redex' \in [\redex]$ such that
$\be^{\redex'}_{k'}=\al$ and $\be^{\redex'}_{k'+1}=\be$ for some
$k'\in\{1, \cdots, \ell (w)-1\}$.

Observe that the following property is followed by the algorithm of
combinatorial AR quivers
\[ \{\, i_t \,|\,  k<t<l, \,  i_t= i\}= \{\, i_t \,|\,  k<t<l, \,  i_t= j\}=\emptyset.\]
Hence the direct path from $\beta$ to $\alpha$ is the only path
starting from $\beta$ to $\alpha$. Otherwise, we get a loop in
$\Delta$ as a consequence, which is a contradiction.

Now let
\[ P=\left\{ a_t \in \N \left|\,
\begin{array}{l}
{\rm (i)}\  k<a_t<l,\, t=1, \cdots, m, \, {\rm (ii)}\ a_1<a_2<\cdots<a_m, \\
{\rm (iii)}\text{ each } \beta_{a_t}^\redex \text{ is on a path to
$\alpha$ in $\Upsilon_{[\redex]}$}.
\end{array}
\right. \right\}  \] and let $P^c= \{ k, k+1, \cdots, l\} \backslash
P=\{b_1, b_2, \cdots, b_{(l-k+1)-m}\}$ where $b_1<b_2<\cdots
<b_{(l-k+1)-m}.$ By our observation, there is no path from
$\beta^\redex_{b_{t'}}$ to $\beta^\redex_{a_t}$ for any $t=1,
\cdots, m$ and $t'=1, \cdots, l-k-m+1.$

Hence $\beta_{b_1}^\redex$ is not connected with any of vertices in
$\{\, \beta_{i}^\redex \, | \, k\leq i <b_1\,\} \subset  \{ \,
\beta_{a_t}^\redex\, | \, t=1, \cdots, m\, \} $ and
\[\redex_{k\leq \cdot \leq l}= (s_{i_k}, \cdots,  s_{i_{{b_1}-1}},s_{i_{b_1}}, s_{i_{{b_1}+1}},\cdots, s_{i_l})\sim (s_{i_{b_1}}, s_{i_k}, \cdots, s_{i_{{b_1}-1}},s_{i_{{b_1}+1}},\cdots,  s_{i_l}).\]
Inductively, we can do the same thing with $b_2, \cdots,
b_{(l-k+1)-m}$ and finally get the following equivalent reduced
expression to $\redex_{k\leq \cdot \leq l}$:
\[ \redex_{k\leq \cdot \leq l}\sim (s_{i_{b_1}}, \cdots, s_{i_{b_{(l-k+1)-m}}}, s_{i_k}, s_{i_{a_1}}, \cdots, s_{i_{a_m}}, s_{i_l}).\]
Since $\beta_{a_m}^\redex$ is not connected to $\beta$, we have
$s_{i_{a_m}} s_{i_l}=s_{i_l}s_{i_{a_m}}.$ Hence
\[ \redex_{k\leq \cdot \leq l}\sim (s_{i_{b_1}}, \cdots, s_{i_{b_{(l-k+1)-m}}}, s_{i_k}, s_{i_{a_1}}, s_{i_{a_{m-1}}}, \cdots, s_{i_l},  s_{i_{a_m}}).\]
Inductively, we get
\[ \redex_{k\leq \cdot \leq l}\sim (s_{i_{b_1}}, \cdots, s_{i_{b_{(l-k+1)-m}}}, s_{i_k}, s_{i_l}, s_{i_{a_1}}, \cdots, s_{i_{a_m}} ).\]
Let  $\redex'=(s_{i'_1}, \cdots, s_{i'_{\ell(w)}})$ have the form
\[\redex'=(s_{i_1}, \cdots, s_{i_{k-1}}, s_{i_{b_1}}, \cdots, s_{i_{b_{(l-k+1)-m}}}, s_{i_k}, s_{i_l}, s_{i_{a_1}}, \cdots, s_{i_{a_m}}, s_{i_{l+1}}, \cdots, s_{i_{\ell(w)}}).\]
Then $s_{i'_{l-m+1}}=s_{i_k}$  (resp. $s_{i'_{l-m+2}}=s_{i_l}$) and
$\beta^{\redex'}_{l-m+1}=\beta^\redex_k=\alpha$ (resp.
$\beta^{\redex'}_{l-m+2}=\beta^\redex_l=\beta$). Hence we proved our
assertion by setting $k'=l-m+1.$
\end{proof}

\begin{proposition}  \label{Prop:meaning of number in arrow}
Let $\alpha$ and $\beta$ have residues $i=i_0$ and $j=i_k$ in $\Upsilon_{[\redex]}$.
If $\alpha$ and $\beta$ are in a sectional path
$$ \beta=\gamma_k \xrightarrow{{m_{i_{k-1},i_k}}}  \gamma_{k-1} \xrightarrow{{m_{i_{k-2},i_{k-1}}}}  \cdots
\xrightarrow{\hspace*{0.7cm}} \gamma_{1} \xrightarrow{m_{i_0,i_{1}}} \gamma_{0}=\alpha $$
in $\Upsilon_{[\redex]}$, then we have
\begin{equation} \label{Eqn:2.3}
(\alpha, \beta)  = \left\{
\begin{array}{ll}
\displaystyle \prod_{t=1}^{k-1}  2^{\delta_{3, i_t}}  \prod_{t=0}^{k-1} m_{i_t,i_{t+1}} & \text{ for Type $F_4$}, \\
\displaystyle \prod_{t=0}^{k-1} m_{i_t,i_{t+1}}  & \text{ otherwise},
 \end{array}
 \right.
\end{equation}
where $i_t$ is the residue of $\gamma_t$ and $m_{a,b} \seteq -(\alpha_a,\alpha_b)$ for $a,b \in I$ (Algorithm \ref{Alg_AbsAR}). Hence
\[ (\alpha, \beta)>0.\]
\end{proposition}

\begin{proof}
Note that, by induction on $k$, we can see that
\[s_{i_0}s_{i_1}\cdots s_{i_{k-1}}(\alpha_{i_k})=  \alpha_{i_k} + \sum_{p=1}^k (-2)^p\frac{\prod_{t=0}^{p-1}(\alpha_{i_{k-t-1}}, \alpha_{i_{k-t}})}{\prod_{t=0}^{p-1}(\alpha_{i_{k-t-1}}, \alpha_{i_{k-t-1}})} \, \alpha_{i_{k-p}}.\]

Proposition \ref{the prop} shows that there exists $w \in W$ such that
\[ \alpha= w (\alpha_i) \text{ and } \beta= w s_i s_{i_1} s_{i_2}\cdots s_{i_{k-1}}(\alpha_j).\]
We have
\begin{equation}
\begin{aligned}
& (w (\alpha_i), w s_i s_{i_1} s_{i_2}\cdots s_{i_{k-1}}(\alpha_j)) = (\alpha_{i_0}, s_{i_0} s_{i_1} s_{i_2}\cdots s_{i_{k-1}}(\alpha_j)) \\
&  =\left(\,  \alpha_{i_0} \, , \alpha_{i_k} + \sum_{p=1}^k (-2)^p\frac{\prod_{t=0}^{p-1}(\alpha_{i_{k-t-1}}, \alpha_{i_{k-t}})}{\prod_{t=0}^{p-1}(\alpha_{i_{k-t-1}}, \alpha_{i_{k-t-1}})} \, \alpha_{i_{k-p}} \, \right). \\
\end{aligned}
 \end{equation}
 Since $(\alpha_{i_0}, \alpha_{i_p})= 0$  for $p\neq 0,1$,
 \begin{equation}
 \begin{aligned}
  & (\alpha_{i_0}, s_{i_0} s_{i_1} s_{i_2}\cdots s_{i_{k-1}}(\alpha_j))\\
  & =  \left(\alpha_{i_0},   (-2)^{k-1}\frac{\prod_{t=1}^{k-1}(\alpha_{i_{t}}, \alpha_{i_{t+1}})}{\prod_{t=1}^{k-1}(\alpha_{i_{t}}, \alpha_{i_{t}})} \, \alpha_{i_{1}}+ (-2)^{k}\frac{\prod_{t=0}^{k-1}(\alpha_{i_{t}}, \alpha_{i_{t+1}})}{\prod_{t=0}^{k-1}(\alpha_{i_{t}}, \alpha_{i_{t}})} \, \alpha_{i_{0}}  \right) \\
& =    - (-2)^{k-1}\frac{\prod_{t=1}^{k-1}(\alpha_{i_{t}}, \alpha_{i_{t+1}})}{\prod_{t=1}^{k-1}(\alpha_{i_{t}}, \alpha_{i_{t}})} \, ( \alpha_{i_0},\alpha_{i_{1}} ) = \prod_{t=1}^{k-1} \frac{2}{(\alpha_{i_t},\alpha_{i_t})}\prod_{t=0}^{k-1} -(\alpha_{i_{t}}, \alpha_{i_{t+1}}).
 \end{aligned}
 \end{equation}
 Here we note that the residue $i_t$ for $t=1,2,\cdots, k-1$ cannot be $1$ or the rank $n$. (Only $i_0$ and $i_k$ can be $1$ or $n$.) According to \cite{Bour}, in the case of except $F_4$, we can check that $(\alpha_{i_t},\alpha_{i_t})=2$ for all $t=1,2\cdots, k-1$. In the case of type $F_4$, we have $(\alpha_2, \alpha_2)=2$ and $(\alpha_3, \alpha_3)=1.$ Hence we get the formula (\ref{Eqn:2.3}).
\end{proof}


\begin{remark}
For any finite type other than $F_4$, we have
\[ (\alpha, \beta)=\prod_{t=0}^{k-1} (\gamma_t, \gamma_{t+1})= \prod_{t=0}^{k-1} -(\alpha_{i_t}, \alpha_{i_{t+1}}) = \prod_{t=0}^{k-1} m_{i_t, i_{t+1}}>0. \]
Here we use notations in Proposition \ref{Prop:meaning of number in arrow}.
\end{remark}

\begin{example} \label{Ex: C3}
Let us consider  $\redez=(s_3,s_2, s_3, s_2, s_1, s_2, s_3, s_2, s_1)$ of type $C_3$. Then
$$\Upsilon_{[\redez]}=
 \scalebox{0.84}
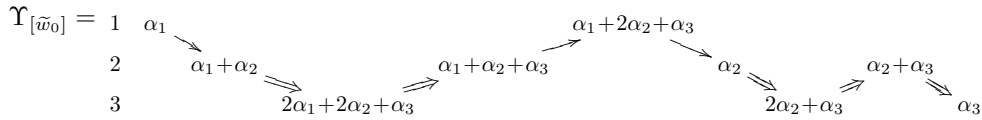
{\xymatrix@C=1ex@R=1ex{
1& \alpha_1 \ar@{->}[dr] &&&&\alpha_1 \shp 2\alpha_2 \shp \alpha_3\ar@{->}[dr]&&&&\\
2&& \alpha_1 \shp \alpha_2 \ar@{=>}[dr] && \alpha_1 \shp \alpha_2 \shp \alpha_3\ar@{->}[ur] && \alpha_2\ar@{=>}[dr] && \alpha_2 \shp \alpha_3\ar@{=>}[dr]& \\
3&&& 2\alpha_1 \shp 2\alpha_2 \shp \alpha_3\ar@{=>}[ur] &&&& 2\alpha_2
\shp \alpha_3\ar@{=>}[ur] && \alpha_3 }}.
$$
One can check that Proposition \ref{Prop:meaning of number in arrow} holds in the above quiver. For instance,
\begin{align*}
2 & = (\alpha_1 + 2\alpha_2 + \alpha_3,2\alpha_1 + 2\alpha_2 + \alpha_3)  \\ & =(\alpha_1 + 2\alpha_2 + \alpha_3,\alpha_1 + \alpha_2 + \alpha_3)(\alpha_1 + \alpha_2 + \alpha_3,2\alpha_1 + 2\alpha_2 + \alpha_3) \\
& = (\alpha_1,\alpha_2)(\alpha_2,\alpha_3).
\end{align*}
\end{example}

\begin{proposition} \label{Lemma:converse}
Let $\alpha, \beta \in \Phi(w)$ and $\redex$ be a reduced expression
of $w\in \W$. Suppose there is no path between $\alpha$ and $\beta$
in $\Upsilon_{[\redex]}$. Then we have $(\alpha, \beta)=0$.
\end{proposition}
\begin{proof} 
Since $<_{\redex}$ is a total order, we can assume that
$\be^{\redex}_k=\al$ and $\be^{\redex}_l=\be$ for $k <l$ without
loss of generality. If $l-k =1$, then
\[ (\alpha, \beta)=(s_{i_1}\cdots, s_{i_{k-1}}(\alpha_{i_k}), s_{i_1}\cdots, s_{i_{k-1}}s_{i_k}(\alpha_{i_l}))=(\alpha_{i_k},s_{i_k}(\alpha_{i_l}))=(\alpha_{i_k}, \alpha_{i_l})=0. \]

Now suppose $l-k\geq 2$. It is enough to find $\redex' \in [\redex]$
such that $\be^{\redex'}_{k'}=\al$ and $\be^{\redex'}_{k'+1}=\be$
for some $k'\in \{1, \cdots, \ell(w)-1\}$. Take the set
\[ P=\left\{ a_t \in \N \left|\,
\begin{array}{l}
 k<a_t<l,\, t=1, \cdots, m, \,  a_1<a_2<\cdots<a_m, \\
\text{ each } \beta_{a_t}^\redex \text{ is on a path to $\alpha$ in
$\Upsilon_{[\redex]}$}.
\end{array}
\right. \right\}  \] and let $P^c= \{ k, k+1, \cdots, l\} \backslash
P=\{b_1, b_2, \cdots, b_{(l-k+1)-m}\}$ where $b_1<b_2<\cdots
<b_{(l-k+1)-m}.$ Then $\beta_{b_1}^\redex$ is not connected with any
of vertices in $\{\, \beta_{i}^\redex \, | \, k\leq i <b_1\,\} $.
Hence
\[\redex_{k\leq \cdot \leq l}= (s_{i_k}, \cdots,  s_{i_{{b_1}-1}},s_{i_{b_1}}, s_{i_{{b_1}+1}},\cdots, s_{i_l})\sim (s_{i_{b_1}}, s_{i_k}, \cdots, s_{i_{{b_1}-1}},s_{i_{{b_1}+1}},\cdots,  s_{i_l}).\]
Inductively, we can do the same thing with $b_2, \cdots,
b_{(l-k+1)-m}$ and finally get the following equivalent reduced
expression to $\redex_{k\leq \cdot \leq l}$:
\[ \redex_{k\leq \cdot \leq l}\sim (s_{i_{b_1}}, \cdots, s_{i_{b_{(l-k+1)-m}}}, s_{i_k}, s_{i_{a_1}}, \cdots, s_{i_{a_m}}, s_{i_l}).\]
Since $\beta_{a_m}^\redex$ is not connected to $\beta$, we have
$s_{i_{a_m}} s_{i_l}=s_{i_l}s_{i_{a_m}}.$ Hence
\[ \redex_{k\leq \cdot \leq l}\sim (s_{i_{b_1}}, \cdots, s_{i_{b_{(l-k+1)-m}}}, s_{i_k}, s_{i_{a_1}}, s_{i_{a_{m-1}}}, \cdots, s_{i_l},  s_{i_{a_m}}).\]
Inductively, we get
\[ \redex_{k\leq \cdot \leq l}\sim (s_{i_{b_1}}, \cdots, s_{i_{b_{(l-k+1)-m}}}, s_{i_k}, s_{i_l}, s_{i_{a_1}}, \cdots, s_{i_{a_m}} ).\]
Then \[\redex'=(s_{i_1}, \cdots, s_{i_{k-1}}, s_{i_{b_1}}, \cdots,
s_{i_{b_{(l-k+1)-m}}}, s_{i_k}, s_{i_l}, s_{i_{a_1}}, \cdots,
s_{i_{a_m}}, s_{i_{l+1}}, \cdots, s_{i_{\ell(w)}})\] is the reduced
expression in $[\redex]$ we desired.
\end{proof}


We record the following corollary, appeared in the proof of the
above proposition, for the later use.

\begin{corollary} \label{cor: imcom adja}
Let $\alpha, \beta \in \Phi(w)$ and $\redex$ be a reduced expression
of $w\in \W$. If there is no path between $\alpha$ and $\beta$ in
$\Upsilon_{[\redex]}$, then there are two distinct reduced
expressions $\redex'$ and $\redex''$ in $[\redex]$ and two integers
$k, l\in \N$ such that $\beta^{\redex'}_{k+1}=\alpha$,
$\beta^{\redex'}_{k}=\beta$ and  $\beta^{\redex''}_{l}=\alpha$,
$\beta^{\redex''}_{l+1}=\beta$.
\end{corollary}

\begin{proposition} \label{Prop:level_AR}
For a reduced expression $\redex$ of $w \in W$ of any finite type,
define a function $\lambda_{\Upsilon_{[\redex]}}:\Phi^+(w)\to \N$ in
an inductive way:
$$ \alpha \mapsto \left\{
\begin{array}{ll}
1, & \text{ if $\alpha$ is a sink in $\Upsilon_{[\redex]}$} \\
\max\{\lambda_{\Upsilon_{[\redex]}}(\beta) \, | \, \alpha\to \beta
\text{ in } \Upsilon_{[\redex]}\}+1, & \text{ otherwise. }
\end{array}
 \right.$$
Then we have $\lambda_{\Upsilon_{[\redex]}} =\lambda_{[\redex]}$.
\end{proposition}

\begin{proof}
Our assertion directly follows from Proposition \ref{the prop} and
Proposition \ref{Lemma:converse}.
\end{proof}

By Proposition \ref{Prop:level_AR} and properties of the level
function $\lambda_{[\redex]}$, we have the following theorem.

\begin{theorem} \label{Thm:AR_class}
Two reduced expressions $\redex$ and $\redex'$ are commutation
equivalent if and only if
$\Upsilon_{[\redex]}=\Upsilon_{[\redex']}$.
\end{theorem}
\begin{proof}
It is enough to show that if $\Upsilon_{[\redex]}=
\Upsilon_{[\redex']}$ then $[\redex]=[\redex']$. However, since we
know that
$\lambda_{[\redex]}=\lambda_{\Upsilon_{[\redex]}}=\lambda_{\Upsilon_{[\redex']}}=\lambda_{[\redex']}$
and  $\lambda_{[\redex]}= \lambda_{[\redex']}$ implies
$[\redex]=[\redex']$, our assertion follows.
\end{proof}

Theorem \ref{Thm:AR_class} implies that we can get every equivalent reduced expression $\redex'$ to $\redex$ by observing $\Upsilon_{[\redex]}$: 


\begin{theorem} \label{Thm: compatible reading}
Every reduced expression of $w$ in $[\redex]$ can be obtained by a compatible reading of $\Upsilon_{[\redex]}.$
\end{theorem}

\begin{remark}
Throughout this section, we actually use residues as labels of $\Upsilon_{[\redex]}^0$, which need not compute.
In Section \ref{Sec: Labeling}, we will suggest an efficient algorithm for labeling of $\Upsilon_{[\redex]}^0$ with positive roots.
%
%
\end{remark}

Now we will show that combinatorial AR
quivers $\Upsilon_{[\redez]}$ can be considered as a generalization of AR quivers $\Gamma_Q$ by providing isomorphism of quivers
$$ \Upsilon_{[\redez]} \simeq \Gamma_Q   \quad \text{ when $[\redez]=[Q]$}.$$


\begin{lemma} \label{Lem:adapted/nonadapted} For a Dynkin quiver $Q$ of type ADE,
let $\redez= (s_{i_1},\cdots ,s_{i_N}) \in [Q]$ and $(i,j)\in I^2$
be a pair of indices with $d_{\Delta}(i,j)=1$. If
$i_k=i_{k'}=i$ for $k<k'$ then there exists $t$ such that $k<t<k'$
and $i_t=j$.
\end{lemma}
\begin{proof}
 Let us denote
 \begin{equation} \label{Eqn:quiver_sink}
 Q_m=s_{i_m}\cdots s_{i_1}Q \qquad\text{ for } m=1, \cdots N.
 \end{equation}
  Then $i_k$ is a sink of $Q_{k-1}$. Suppose $k'=\min\{k_1 \, | \,  k_1>k, \, i_{k_1}=i\} $. If $\{t\, |\, k<t<k', \, i_k=j\}=\emptyset $
  then the arrow between $i$ and $j$ in  $Q_{k'-1}$ has the direction $i\to j$. Hence the vertex $i$ cannot be a sink of $Q_{k'-1}$, which is a contradiction.
\end{proof}

\begin{example} In Example \ref{nonadapted AR}, we can obtain the following reduced expression in $[\redez]$ by compatible reading:
$$(s_1,s_2,s_5,s_3,s_4,s_3,s_1,s_2,s_5,s_1,s_3,s_4,s_3)$$
In the above reduced expression, one can check that Lemma \ref{Lem:adapted/nonadapted} do {\it not} hold, since it is not adapted to any $Q$.
\end{example}

The following corollary directly follows from Lemma
\ref{Lem:adapted/nonadapted}.

\begin{corollary}
Suppose $\redex= (s_{i_1},\cdots, s_{i_{\ell(w)}})$ is a reduced
expression of $w\in \W$ and let $i,j\in I$ be indices with $d_\Delta(i,j)=1$. If there are $k, k'\in \{1, \cdots, \ell(w)\}$ with $k<k'$
such that
\begin{enumerate}
\item[{\rm (1)}] $i_k=i_{k'}=i,$
\item[{\rm (2)}] $\{\, j' \, |\, k<j'<k',\, i_{j'}=j \, \}=\emptyset$.
\end{enumerate}
then $\redex$ is not adapted to any Dynkin quiver.
\end{corollary}

\begin{theorem} \label{thm: Upsilon Q Gamma Q}
For a Dynkin quiver of type $A_n$, $D_n$ or $E_n$, the combinatorial
AR quiver $\Upsilon_{[Q]}$  is isomorphic to the AR quiver $\Gamma_Q$.
\end{theorem}
\begin{proof}
Let us first show that $A(Q) \simeq \Gamma_Q$ is isomorphic to the
combinatorial AR quiver $\Upsilon_{[Q]}$. Suppose the Dynkin quiver
$Q$ has an arrow from $j$ to $i$. Then we have the following
properties.
\begin{enumerate}[(i)]
\item By the construction of $A(Q)$, there is an arrow from $(j, 1)$ to $(i,1)$ in $A(Q),$
\item Recall the definition of the quiver $Q_m$ in (\ref{Eqn:quiver_sink}). Then $i$ is a sink of $Q_{k_i-1}$ and $j$ is a sink of $Q_{k_j-1}$, where
\[ k_i=\text{ min}\{k  \, | \, s_{i_k}= s_i\} < k_j=\text{ min}\{k  \, | \, s_{i_k}= s_j\}.\]
\item $\{k  \, | \, k_i<k<k_j, \, i_k=i\}=\emptyset$.
\end{enumerate}
Here, the reason for (ii) is that, in order to make the vertex $j$ a
sink in the quiver $Q_{k_j-1}$, we need the reflection $s_i$ which
reverses the arrow from $j$ to $i$. Also, Lemma
\ref{Lem:adapted/nonadapted} implies (iii). Using (ii), (iii) and
the construction of the combinatorial AR quiver, we conclude that
there is an arrow from $\beta^\redez_{k_j}$ to $\beta^\redez_{k_i}$
in $\Upsilon_{[\redez]}$ for $\redez \in [Q]$.

Moreover, in the quiver $A(Q)$, there is an arrow  from $(i, m)$ to
$(j, m-1)$ if
\begin{equation} \label{Eqn:2.4_0405}
\#\{s_{i_k}=s_i \, | \,  1\leq k\leq N\}\geq m, \qquad
\#\{s_{i_k}=s_j \, | \,  1\leq k\leq N\}\geq m-1.
\end{equation}
Let
$$ k_{i,m}=\min \left\{
 k'\, \left|
\begin{array}{l}
\text{ there exists a sequence } k_1<k_2<\cdots <k_m=k' \\
\text{ such that }  i_{k_1}=\cdots=i_{k_m}=i
\end{array}\right.
\right\}.
$$
Then by Lemma \ref{Lem:adapted/nonadapted}, we have
$k_{i}=k_{i,1}<k_{j}=k_{j,1}<k_{i,2}<k_{j,2}<k_{i,3}<\cdots$. Hence,
by the construction of $\Upsilon_{[\redez]}$, we have an arrow  from
$\beta^\redez_{k_{i,m}}$ to $\beta^\redez_{k_{j,m-1}}$ if
(\ref{Eqn:2.4_0405}) holds.

Similarly, if we have
\[ \#\{s_{i_k}=s_i \, | \,  1\leq k\leq N\}\geq m, \qquad \#\{s_{i_k}=s_j \, | \,  1\leq k\leq N\}\geq m,\]
then we have the arrow from $(j,m)$ to $(i,m)$ in $A(Q)$ and the
arrow from $\beta^\redez_{k_{j,m}}$ to $\beta^\redez_{k_{i,m}}$ in
$\Upsilon_{[\redez]}$.

As a conclusion, two quivers $A(Q)$ and  $\Upsilon_{[Q]}$ are
isomorphic by the map $\psi:A(Q) \to \Upsilon_{[Q]}$ such that
$(i,m)\mapsto \beta^\redez_{k_{i,m}}$. Recall that the quiver
isomorphism  $\iota_Q:A(Q)\to \Gamma_Q$ was defined by
$(i,m)\mapsto \beta^\redez_{k_{i,m}}$. Hence the ordinary AR quiver
$\Gamma_Q$ and the combinatorial AR quiver $\Upsilon_{[\redez]}$ are
isomorphic to each other.
\end{proof}

Now we can prove that, for any $\redex$ of $w \in W$ of any finite
type, $\Upsilon_{[\redex]}$ visualizes the convex partial order
$\prec_{[\redex]}$ on $\Phi(w)$:

\begin{theorem} \label{Thm:order_path}
The combinatorial AR quiver $\Upsilon_{[\redex]}$ visualizes the
convex partial order $\preceq_{[\redex]}$. That is
$\alpha\preceq_{[\redex]} \beta$ if and only if there is a path from
$\beta$ to $\alpha$ in $\Upsilon_{[\redex]}$.
\end{theorem}

\begin{proof}
It is obvious that then there is a path from $\beta$ to $\alpha$ in
$\Upsilon_{[\redex]}$ then we have $\alpha \prec_{[\redex]} \beta$.

Conversely, suppose that there is no path between $\beta$ and
$\alpha$ and $\pi_\redex(\beta)<\pi_\redex(\alpha)$. In the proof of
Proposition \ref{Lemma:converse}, we showed that there is a reduced
expression $\redex'=(s_{i_1}, \cdots, s_{i_{\ell(w)}})\in [\redex]$
such that $\alpha=\beta^{\redex'}_{k+1}$ and
$\beta=\beta^{\redex'}_{k}$ and $(\alpha_{i_k},
\alpha_{i_{k+1}})=0$. Hence by exchanging $s_{i_k}$ and
$s_{i_{k+1}}$ in $\redex'$, we get $\redex''=(s_{i_1}, \cdots,
s_{i_{k-1}},s_{i_{k+1}}, s_{i_k}, s_{i_{k+2}}, \cdots,
s_{i_{\ell(w)}})\in [\redex]$. Since
$\pi_{\redex''}(\beta)=k+1>\pi_{\redex''}(\alpha)=k$, we conclude
that $\beta$ and $\alpha$ are not comparable via
$\preceq_{[\redex]}$.
\end{proof}


\section{Combinatorial reflection maps}\label{sec: reflection map}

The following theorem is a well-known fact about sinks and sources
of a Dynkin quiver $Q$ and an AR quiver $\Gamma_Q$.

 \begin{theorem} Let $Q$ be a Dynkin quiver of type $A_n$, $D_n$, or $E_n$ and $\Gamma_Q$ be the associated AR quiver. The followings are equivalent.
 \begin{enumerate}
 \item[{\rm (a)}] $i\in I$ is a sink $($resp. source$)$ of $Q$.
 \item[{\rm (b)}] There are reduced expressions $\redez$ adapted to $Q$ such that $\redez$ starts $($resp. ends$)$ with $s_i$ $($resp. $s_{i^*})$.
 \item[{\rm (c)}] $\alpha_i$ is a sink $($resp. source$)$ of $\Gamma_Q$.
 \end{enumerate}
 \end{theorem}

Let $X_n$ be a simply laced type, i.e., $X$ is $A,D$ or $E$, and $n$ is a proper integer depending of $X$.
 On the set of AR quiver $\Gamma_{\lf Q \rf}=\{ \, \Gamma_Q \, |\, \text{ $Q$ is a Dynkin quiver of type $X_n$} \}$, for $i\in I$,
define {\it right $($resp. left$)$ reflection map}
 \[ r_i : \Gamma_{\lf Q \rf}\to \Gamma_{\lf Q \rf}\] by   $\Gamma_Q \mapsto \Gamma_Q r_i\ (\text{resp. }\Gamma_Q \mapsto \Gamma_Q r_i)$, where
 \begin{equation}
\Gamma_Q r_i=\left\{
\begin{array}{ll}
 \Gamma_{s_i(Q)} & \text{ if $i$ is a sink  in $Q$,}\\
 \Gamma_{Q} & \text{ otherwise,} \text{ and }
 \end{array}\right.
r_i \Gamma_Q  =\left\{
\begin{array}{ll}
 \Gamma_{s_{i^*}(Q)} & \text{ if $i^*$ is a source  in $Q$,}\\
 \Gamma_{Q} & \text{ otherwise.}
 \end{array}\right.
 \end{equation}

\begin{remark}
The reflection maps in this section can be understood as a combinatorial version of the reflection functors in
the representation theory over the path algebra $\C Q$ of a quiver $Q$ (see \cite[Section 3]{Krause} for reference).
\end{remark}

 \begin{example} Let $\redez=(s_3, s_1, s_2, s_4, s_1, s_3, s_5, s_2, s_4, s_1, s_3, s_5, s_2, s_1, s_4)$ of $A_5$. Note that
 $\redez$ is adapted one. Then $\alpha_3$ is a sink of $\Gamma_{[\redez]}$ and $\alpha_2$ is a source of $\Gamma_{[\redez]}$.
\begin{align*}
&  \scalebox{0.84}
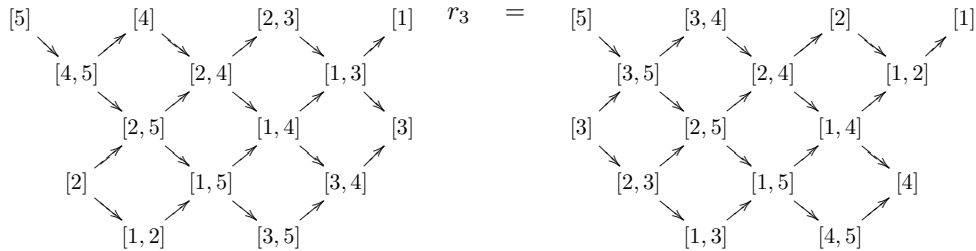
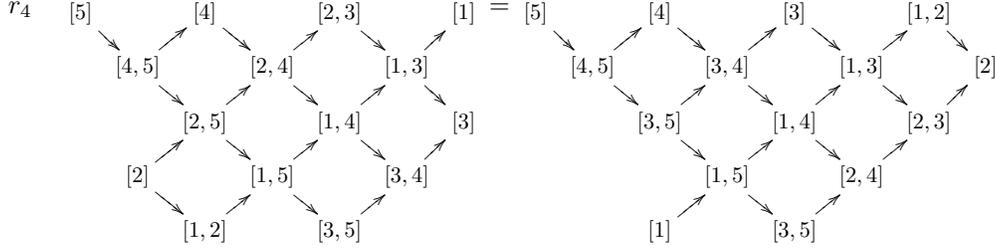
{\xymatrix@C=1ex@R=2ex{
    [5]\ar@{->}[dr]  && [4]\ar@{->}[dr] &&[2,3] \ar@{->}[dr] && [1] \\
 &[4,5] \ar@{->}[ur]\ar@{->}[dr]&&[2,4] \ar@{->}[ur]\ar@{->}[dr]&&[1,3] \ar@{->}[ur]\ar@{->}[dr]\\
 && [2,5]\ar@{->}[ur]\ar@{->}[dr] && [1,4]\ar@{->}[ur]\ar@{->}[dr] && [3] \\
 &  [2] \ar@{->}[ur]\ar@{->}[dr]  && [1,5] \ar@{->}[ur] \ar@{->}[dr] && [3,4] \ar@{->}[ur]\\
 & &[1,2] \ar@{->}[ur] &&[3,5] \ar@{->}[ur] }} \quad  r_3  \quad
 = \scalebox{0.84}{ \xymatrix@C=1ex@R=2ex{
   & [5]\ar@{->}[dr]  && [3,4]\ar@{->}[dr] &&[2] \ar@{->}[dr] && [1] \\
 &&[3,5] \ar@{->}[ur]\ar@{->}[dr]&&[2,4] \ar@{->}[ur]\ar@{->}[dr]&&[1,2] \ar@{->}[ur]\\
&  [3]  \ar@{->}[ur] \ar@{->}[dr]&& [2,5]\ar@{->}[ur]\ar@{->}[dr] && [1,4]\ar@{->}[ur]\ar@{->}[dr] \\
 & &  [2,3] \ar@{->}[ur]\ar@{->}[dr]  && [1,5] \ar@{->}[ur] \ar@{->}[dr] && [4] \\
 && &[1,3] \ar@{->}[ur] && [4,5] \ar@{->}[ur] }} \allowdisplaybreaks\\
 & r_4 \scalebox{0.84}{ \xymatrix@C=1ex@R=2ex{
   & [5]\ar@{->}[dr]  && [4]\ar@{->}[dr] &&[2,3] \ar@{->}[dr] && [1]\\
 &&[4,5] \ar@{->}[ur]\ar@{->}[dr]&&[2,4] \ar@{->}[ur]\ar@{->}[dr]&&[1,3] \ar@{->}[ur]\ar@{->}[dr]\\
& & & [2,5]\ar@{->}[ur]\ar@{->}[dr] && [1,4]\ar@{->}[ur]\ar@{->}[dr] && [3]\\
 & &  [2] \ar@{->}[ur]\ar@{->}[dr]  && [1,5] \ar@{->}[ur] \ar@{->}[dr] && [3,4] \ar@{->}[ur]\\
 && &[1,2] \ar@{->}[ur] &&[3,5] \ar@{->}[ur]}}
 = \scalebox{0.84}{\xymatrix@C=1ex@R=2ex{
    [5]\ar@{->}[dr]  && [4]\ar@{->}[dr] &&[3] \ar@{->}[dr] && [1,2]  \ar@{->}[dr]\\
 &[4,5] \ar@{->}[ur]\ar@{->}[dr]&&[3,4] \ar@{->}[ur]\ar@{->}[dr]&&[1,3] \ar@{->}[ur]\ar@{->}[dr] && [2] \\
 && [3,5]\ar@{->}[ur]\ar@{->}[dr] && [1,4]\ar@{->}[ur]\ar@{->}[dr] && [2,3]  \ar@{->}[ur]\\
 &&& [1,5] \ar@{->}[ur] \ar@{->}[dr] && [2,4] \ar@{->}[ur]\\
 &&[1] \ar@{->}[ur] &&[3,5] \ar@{->}[ur]}}
\end{align*}
 \end{example}
Let $i$ be a sink (resp. source) in $Q$. The right (resp. left) reflection map $r_i$ on $\Gamma_Q$ can be described as follows: 
\begin{eqnarray} &&
\parbox{103ex}{
\begin{itemize}
\item[{\rm (i)}] Delete the sink (resp. source) $\al_i$ (resp. $\al_{i^*}$) in $\Gamma_Q$.
\item[{\rm (ii)}] Put a new vertex $\al_i$ (resp. $\al_{i^*}$) with residue $i^*$  at the beginning (resp. end) of $\Gamma_{Q}$ and arrows starting from $\al_i$ (resp. ending at  $\al_{i^*}$)  and ending at the first vertices (resp. starting from the last vertices) with residues $j$ such that $d_{\Delta}(i^*,j)=1$.


\item[{\rm (iii)}] Change each label $\beta$ in $\Phi^+ \setminus \{ \al_i\}$ (resp. $\Phi^+ \setminus \{ \al_{i^*}\}$) with $s_i \beta$ (resp. $s_{i^*} \beta$).
\end{itemize}
}\label{ref Q h}
\end{eqnarray}

\vskip 3mm

Analogously, we can define reflection maps on combinatorial AR
quivers. In order to do this, we need notions of source and sink of
commutation classes $[\redex]$ of $\W$.

\begin{definition}
For a commutation equivalence class $[\redex]$, we say that $i\in I$
is a {\it sink} (resp. {\it source}) if there is a reduced
expression $\redex'\in [\redex]$ of $w$ starting with $s_i$  (resp.
ending with $s_i$).
\end{definition}

The following proposition follows from the construction of the
combinatorial AR quiver $\Upsilon_{[\redex]}$ and \eqref{eq:
cahracterization}:

\begin{proposition} \hfill
\begin{enumerate}
\item[{\rm (a)}]  $i$ is a sink of $[\redex]$ if and only if $\alpha_i$ is a sink in the quiver $\Upsilon_{[\redex]}$.
\item[{\rm (b)}] $i$ is a source of $[\redex]$ if and only if  $-w(\alpha_i)$ is a source in the quiver $\Upsilon_{[\redex]}$.
\end{enumerate}
\end{proposition}

Using sources and sinks of a commutation equivalence class, we shall
define a reflection map on the set of combinatorial AR quivers
$$\Upsilon_{w_0} \seteq \{ \, \Upsilon_{[\redez]} \, |\, \redez\text{ is a reduced expression of $w_0$ }\}$$
and divide the set $\Upsilon_{w_0}$ into the orbits
$\Upsilon_{\lf\redez \rf}$ of reflection maps (see also
Definition \ref{def: ref equi} below):
$$\Upsilon_{w_0} = \bigsqcup_{\lf \redez \rf} \Upsilon_{\lf\redez \rf}$$

\begin{definition}
The right reflection map $r_i$ on $[\redez]$ is
defined by
$$[\redez]\, r_i =
\left\{
\begin{array}{ll}
[(s_{i_2},\cdots, s_{i_N}, s_{i^*})] & \text{ if $i$ is a sink and } \redez'=(s_{i}, s_{i_2}, \cdots, s_{i_N})\in [\redez],\\
\ [\redez] &  \text{ if $i$ is not a sink of $[\redez]$}.
\end{array}
\right.
$$
On the other hand, the left reflection map $r_i$
on $[\redez]$ is defined by
$$
r_i \, [\redez]= \left\{
\begin{array}{ll}
[(s_{i^*},s_{i_1}\cdots, s_{i_{N-1}})] &  \text{ if $i$ is a source and } \redez'=(s_{i_1}, \cdots, s_{i_{N-1}}, s_i)\in [\redez],\\
\ [\redez] &  \text{ if $i$ is not a source of $[\redez]$}.
\end{array}
\right.
$$
\end{definition}

\begin{lemma}
For a reduced expression $\redez=(s_{i_1},\ldots, s_{i_N})$ of
$w_0$, define $\redex_{< k} = s_{i_1}\cdots s_{i_{k-1}}$ for $1 \le
k \le N$. Then
$$ \redex_{< N}(\alpha_{i_N}) = \al_{i_{N}^*}.$$
\end{lemma}

\begin{proof}
$ \redex_{< N}(\alpha_{i_N}) = w_0 \cdot s_{i_N} (\al_{i_N})=
w_0(-\al_{i_N}) = \al_{i^*_N}$.
\end{proof}

The following propositions show the reflection map is
well-defined on
$$\{ \, [\redez] \, |\, \redez\text{ is a reduced expression of $w_0$ }\}.$$

\begin{proposition} \label{prop: generalized reflection}
Let $\redez=(s_{i_1}, \cdots, s_{i_{N-1}},s_{i_N})$ be a reduced
expression of $w_0$.
\begin{enumerate}
\item[{\rm (a)}] $\redez'=(s_{i^*_N},s_{i_1}, \cdots, s_{i_{N-1}})$ is a reduced expression of $w_0$ which is not in $[\redez]$.
\item[{\rm (b)}] $\redez''=(s_{i_2},\cdots, s_{i_{N-1}},s_{i_N},s_{i_1^*})$ is a reduced expression of $w_0$ which is not in $[\redez]$.
\end{enumerate}
\end{proposition}

\begin{proof}
{\rm (a)} Suppose $\redez'$ is not a reduced expression. Then
$\redez'$ represents $w \in \W$ whose length is $N-2$, that is
\[w=s_{i^*_N}s_{i_1} \cdots s_{i_{N-1}}\in \W \quad \text{ and }
\quad  \ell(w)=N-2.\]
 Denote $\beta^\redez_k=s_{i_1}\cdots s_{i_{k-1}}\alpha_k$ for $k=1, \cdots, N$. Recall that $w_0^{-1} \beta^\redez_k\in \Phi^-$
 and $|\, \{\beta^\redez_k\, |\, 1\leq k\leq N\}\, |=|\Phi^+|=N$. Since $\beta^\redez_N=\alpha_{i_N^*}$ and
 $|\, \Phi(w)\, |=|\,  \{\alpha\in \Phi^+| w^{-1}(\alpha)\in \Phi^-\}\, |=N-2, $ there exists $m$ such that $1\leq m\leq N-1$
 and  $w^{-1}\beta^\redez_m \in \Phi^+\backslash\{\alpha_{i_N^*}\}$. Observe that $w^{-1}  \beta^\redez_m=s_{i_N}w_0^{-1}s_{i^*_N}\beta^\redez_m\in \Phi^+$
 and $w_0^{-1}(s_{i^*_N}\beta^\redez_m)=w_0(s_{i^*_N}\beta^\redez_m)\in \Phi^-$.  Hence $w_0^{-1}s_{i_N^*}\beta^\redez_m=s_{i_N}w^{-1}\beta^\redez_m=-\alpha_{i_N}$
 and $w^{-1}\beta^\redez_m=\alpha_{i_N}$. Now we get
 \[\beta^\redez_m=w \alpha_{i_N}=s_{i_N^*}s_{i_1}\cdots s_{i_{N-1}}\alpha_{i_N}=s_{i_N^*}(\alpha_{i_N^*})=-\alpha_{i_N^*}, \]
  which is a contradiction.  Hence the length of $w$ is $N$. In other words, $w=w_0$ and $\redez'$ is a reduced expression of $w_0$.

Moreover, we have $[\redez]\neq[\redez']$ since $\la_{\redez}(\al_{i^*_N}) >1$  and $\la_{\redez'}(\al_{i^*_N})=1$. \\
{\rm (b)} The analogous proof of (a) works to (b).
\end{proof}

\begin{remark}
To the experts, the fact that $\redez'$ and $\redez''$ are also reduced expressions of $w_0$ may be well-known (for example, \cite[page 7]{CT15} and \cite[page 650]{Kato12}). However, we have had a difficulty finding its proof.
Thus we provide a proof by using the system of positive roots.
\end{remark}

\begin{proposition} \label{Prop:well-def reflec}
Let $\redez=(s_{i_1}, \cdots, s_{i_N})$ and $\redez'=(s_{i'_1},
\cdots s_{i'_N})$ be reduced expressions in $[\redez]$.
\begin{enumerate}
\item[{\rm (a)}] If $i_1=i'_1$ then $\redez^1=(s_{i_2}, \cdots, s_{i_N}, s_{i^*_1})$ and $\redez^2=(s_{i'_2}, \cdots, s_{i'_N}, s_{i^*_1})$
are in the same commutation equivalence class.
\item[{\rm (b)}] If $i_N=i'_N$ then $\redez^3=(s_{i^*_N}, s_{i_1}, \cdots, s_{i_{N-1}})$ and  $\redez^4=(s_{i^*_N}, s_{i'_1}, \cdots, s_{i'_{N-1}})$
are in the same commutation equivalence class.
\end{enumerate}
\end{proposition}

\begin{proof}
Since we proved $\redez^p$, $p=1,2,3,4$, are all reduced expressions
of $w_0$, it is enough to show that
$\Upsilon_{[\redez^1]}=\Upsilon_{[\redez^2]}$ and
$\Upsilon_{[\redez^3]}=\Upsilon_{[\redez^4]}$. If $i_1=i'_1$ then
$\Upsilon_{[\redex]}$ and $\Upsilon_{[\redex']}$ are the same
subquiver of $\Upsilon_{[\redez]}$ where $\redex=(s_{i_2}, \cdots,
s_{i_N})$ and $\redex=(s_{i'_2}, \cdots, s_{i'_N})$. By the
algorithm of combinatorial AR quivers, there is a unique way to get
another combinatorial AR quiver by putting the last vertex on the
$i_1^*$-th residue of $\Upsilon_{[\redex]}$. Since the resulting AR
quiver is $\Upsilon_{[\redez^1]}=\Upsilon_{[\redez^2]}$, we have
$[\redez^1]=[\redez^2]$. Similarly, we can show that
$[\redez^3]=[\redez^4]$.
\end{proof}

The reflecting map on $[\redez]$ induces the    right (resp.
left)  {\it reflection map} $r_i$ for $i\in I$ on
$\Upsilon_{w_0} $ as follows:
\begin{equation}
\Upsilon_{[\redez]} \, r_i=
 \Upsilon_{[\redez] \,  r_i}  \qquad  (\text{resp. } r_i\,  \Upsilon_{[\redez]} = \Upsilon_{r_i[\redez]} ).
 \end{equation}

If  $s_{i_1}$ is a sink in $[\redez]$, there is a reduced expression
$\redez=(s_{i_1}, s_{i_2}, \cdots, s_{i_N})\in [\redez]$. We know
that $[\redez] s_{i_1}= [\redez']$ where $\redez'=(s_{i'_1},
s_{i'_2}, \cdots, s_{i'_N})=(s_{i_2},\cdots, s_{i_{N}},s_{i_1^*})$.
Hence  for $k=1, \cdots, l-1$, we have
$\beta^{\redez'}_k=s_{i_1}\beta^{\redez}_{k+1}$.

 When $s_{i_N}$ is a source in $[\redez]$, there is a reduced expression $\redez=(s_{i_1}, s_{i_2}, \cdots, s_{i_N})\in [\redez]$.
 We know that $s_{i_N}[\redez]= [\redez']$ where $\redez'=(s_{i'_1}, s_{i'_2}, \cdots, s_{i'_N})=(s_{i_N^*},s_{i_1},\cdots, s_{i_{N-1}})$.
 Hence  for $k=2, \cdots, l$, we have $\beta^{\redez'}_k=s_{i_N^*}\beta^\redez_{k-1}$.

For $\redex=(s_{i_1}, \cdots, s_{i_k})$, the right  (resp. left)
action of the reflection map  $r_{\redex}$ is defined by
$$ [\redez] \, r_{\redex}=[\redez] r_{i_1} \cdots r_{i_k} \qquad (\text{resp. } r_{\redex}[\redez]=r_{i_k}\cdots r_{i_1} [\redez]).$$

 Then the right (resp. left) reflection map on $\Upsilon_{[\redez]}$ can be described as an analogue of \eqref{ref Q h}:
\begin{eqnarray} &&
\parbox{103ex}{
\begin{itemize}
\item[{\rm (i)}] Delete the sink (resp. source) $\al_i$ (resp. $\al_{i^*}$) with residue $i$ (resp. $i^*$) and arrows
incident with $\al_i$ (resp. $\al_{i^*}$) in $\Upsilon_{[\redez]}$.
\item[{\rm (ii)}] Put a new vertex $\al_i$ (resp. $\al_{i^*}$) in the end (resp. beginning) of $\Upsilon_{[\redez]}$ and arrows the conditions in Algorithm \ref{Alg_AbsAR}.
\item[{\rm (iii)}] Change each label $\beta$ in $\Phi^+ \setminus \{ \al_i\}$ (resp. $\Phi^+ \setminus \{ \al_{i^*}\}$) with $s_i \beta$ (resp. $s_{i^*} \beta$).
\end{itemize}
}\label{ref redez h}
\end{eqnarray}

\begin{example} \label{ex: relfection map}
Let us consider reduced expression $\redez=(s_1, s_2, s_1, s_3, s_4,
s_3, s_2, s_3, s_1, s_2)$ of $A_4$ which is not adapted to any
Dynkin quiver $Q$. Then we have
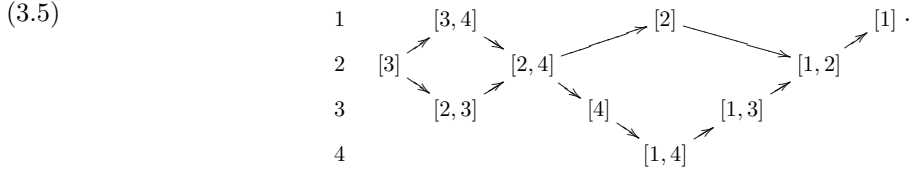
\begin{equation} \label{A4 non-adapted A}
\scalebox{0.84}{\xymatrix@C=2ex@R=1ex{
1 && [3,4] \ar@{->}[dr] &&& [2] \ar@{->}[drr] &&& [1] \\
2 &[3] \ar@{->}[ur]\ar@{->}[dr] && [2,4] \ar@{->}[dr]\ar@{->}[urr] &&&& [1,2] \ar@{->}[ur] \\
3 && [2,3] \ar@{->}[ur] && [4] \ar@{->}[dr] && [1,3] \ar@{->}[ur]\\
4 &&&&& [1,4] \ar@{->}[ur] \\
}}.
\end{equation}
Since $s_2$ is a source of $\redez$, we have $r_2 [\redez]=(s_3,
s_1, s_2, s_1, s_3, s_4, s_3, s_2, s_3, s_1) $ and $r_2
\Upsilon_{[\redez]}$ is
\begin{equation} \label{A4 non-adapted B}
\scalebox{0.84}{\xymatrix@C=2ex@R=1ex{
1 & [4] \ar@{->}[dr] &&& [2,3] \ar@{->}[drr] &&& [1] \\
2  && [2,4] \ar@{->}[dr]\ar@{->}[urr] &&&& [1,3] \ar@{->}[ur]\ar@{->}[dr] \\
3 & [2] \ar@{->}[ur] && [3,4] \ar@{->}[dr] && [1,2] \ar@{->}[ur] && [3]\\
4 &&&& [1,4] \ar@{->}[ur] \\
}}.
\end{equation}

\end{example}

\begin{definition} \label{def: ref equi} \hfill
\begin{enumerate}
\item Let  $[\redez]$ and $[\redez']$ be two commutation equivalence classes. We say $[\redez]$ and $[\redez']$ are
{\it reflection equivalent} and write $[\redez]\overset{r}{\sim}
[\redez']$ if $[\redez']$ can be obtained from $[\redez]$ by a
sequence of reflection maps. The family of commutation
equivalence classes $\lf \redez \rf \seteq \{\, [\redez]\, |\,
[\redez]\overset{r}{\sim}[\redez']\, \}$ is called an {\it r-cluster
point}.
\item If  $[\redez]\overset{r}{\sim}[\redez']$ then we say $\Upsilon_{[\redez]}$ and $\Upsilon_{[\redez']}$ are {\it reflection equivalent}
and write
$\Upsilon_{[\redez]}\overset{r}{\sim}\Upsilon_{[\redez']}$. Also,
$\Upsilon_{\lf \redez \rf} \seteq \{\, \Upsilon_{[\redez]}\, |\,
[\redez]\overset{r}{\sim}[\redez']\, \}$ is called an {\it r-cluster
point}.
\end{enumerate}
\end{definition}

\medskip

Now we shall observe what the equivalent classes in the same
r-cluster point share:


\begin{definition} \label{def: sigma-composition}
Let $\sigma$ be a Dynkin diagram automorphism and $k$ be the number of $\sigma$-orbits of the index set $I$. Take a sequence of $\sigma$-orbits $\mathcal{O}=(o_1, o_2, \cdots, o_k)$ where $o_i\neq o_j$ for $1\leq i< j\leq k.$
For a reduced expression $\redez=(s_{i_1}, \cdots, s_{i_N})$ of $w_0$, the {\it $\sigma$-composition} of $[\redez]$ associated to $\mathcal{O}$ is \[(\mathtt{c}_1, \mathtt{c}_2, \cdots,
\mathtt{c}_k)\in \N^k \qquad \text{ where } \ \mathtt{c}_j=| \{ s_{i_t} \ | \ i_t \in o_j \ \text{ for some } k \in \Z \}|.\]
\end{definition}

\begin{remark}
Depending on the sequence of orbits $\mathcal{O}$, we get different $\sigma$-composition. However, every $\sigma$-composition is same up to order of components. Hence, we fix  $\mathcal{O}=(o_1, o_2, \cdots, o_k)$ satisfying
\[ \, [ \, \text{smallest element in } \, o_i\, ] < [ \, \text{smallest element in }\, o_{i+1}\, ] \, \]
for all $i=1\cdots, k-1.$
\end{remark}



The well definedness of $\sigma$-composition follows by the fact
that if $\redez=(s_{i_1}, \cdots, s_{i_N})$ and $\redez'=(s_{i'_1},
\cdots, s_{i'_N})$ are in the same commutation class then
$$ \#\{ i_k\, |\, i_k = i \, \}= \#\{ i'_k\, |\, i'_k=i  \, \} \text{ for any $i \in I$}.$$

\begin{example} (1) In Example \ref{ex: relfection map}, the $^*$-composition of $[\redez]$ in \eqref{A4 non-adapted A} is
$$(4,6)$$
since there are $4$-many $s_i$ for $i=1$ or $4$ in $\redez$ and $6$-many $s_j$ for $j=2$ or $3$ in $\redez$.

(2) Let us take a Dynkin diagram involution $\sigma$ of $D_4$ as $\sigma(i)=i$ for $1 \le i \le 2$ and $\sigma(3)=4$. Then $\sigma$-composition of $[\redez]$ in Example \ref{D4 non-adapted D-1} is
$$ (4,4,4).$$

(3) Let us take a Dynkin diagram automorphism $\sigma$ of $D_4$ as $\sigma(2)=2$, $\sigma(1)=3$, $\sigma(3)=4$ and $\sigma(4)=1$. Then $\sigma$-composition of $[\redez]$ for
$\redez=(s_1,s_2,s_3,s_2,s_1,s_2,s_4,s_2,s_1,s_2,s_3,s_2)$ is
$$ (6,6).$$
\end{example}

\begin{proposition}
If  two commutation equivalence classes $[\redez]$ and $[\redez']$
of $w_0$ are in the same $r$-cluster point then
$\sigma$-compositions of $[\redez]$ and $[\redez']$ are same.
\end{proposition}

\begin{proof}
Note that any Dynkin diagram automorphism $\sigma$ is compatible with $^*$, it is enough
to show when $\sigma= {}^*$. Let $\redez=(s_{i_1},\cdots, s_{i_N})$.
The only thing we need to show is that  $^*$-compositions of
$[\redez]$, $r_{i_N}[\redez]$ and $[\redez] r_{i_1}$ are same. If
$r_{i_N}[\redez]=[\redez']$  then $(s_{i^*_N},s_{i_1},\cdots,
s_{i_{N-1}})\in[\redez']$. Hence $^*$-compositions of $[\redez]$ and
$[\redez']$ are same. Similarly, $^*$-compositions of
$[\redez]r_{i_1}$ and $[\redez]$ are same. Hence we proved the
proposition.
\end{proof}

\begin{example} \label{Ex:coxeter composition} \hfill
\begin{enumerate}
\item Let $\redez$ be a reduced
expression of $w_0$ of $A_{n}$ adapted to
$$ Q =\xymatrix@R=0.5ex@C=5ex{ *{\circ}<3pt> \ar@{<-}[r]_<{1}  &*{\circ}<3pt>
\ar@{<-}[r]_<{ 2}  &   {} \ar@{.}[r] & *{\circ}<3pt>
\ar@{<-}[r]_>{ \ \ n-1} &*{\circ}<3pt>\ar@{<-}[r]_>{
\quad n} &*{\circ}<3pt> } $$ Let $\sigma=^*$. Then the
$\sigma$-composition of $[\redez]$ consists of $\lceil
\frac{n+1}{2}\rceil$ components such that
\begin{equation}
\left\{
\begin{array} {ll}
(n+1, \cdots, n+1)& \text{ if $n$ is even }\\
(n+1, \cdots, n+1, \frac{n+1}{2}) & \text{ if $n$ is odd}.
\end{array} \right.
\end{equation}
It is well known that all the adapted reduced expressions of $w_0$
are in this r-cluster point and all of equivalent classes in this
r-cluster point are adapted to some Dynkin quiver.
\item In type $A_k$ $(k \le 4)$ (resp. $D_4$), there is a one-to-one correspondence between $r$-cluster points and $^*$-compositions (resp. $\sigma$-composition). ( See  Appendix \ref{Sec:Appen_A4} for $A_4$. )
\item In type $A_5$, there are at least 8 r-cluster points.
\end{enumerate}
\end{example}

\begin{remark}
The number of commutation classes of reduced expressions  $[\redez]$ of type $A_n$ increases exponentially as the $n$ increases \cite[A006245]{OEIS}.
However, in this paper, we claim that commutation classes in the same cluster point are closely related to each other. Hence classifying cluster points can be an interesting problem. (See Conjecture \ref{Conjecture} below.)
\end{remark}

\begin{conjecture} \label{Conjecture}
As a generalization of Example \ref{Ex:coxeter composition} (2), we conjecture that the $\sigma$-Coxeter composition classifies all cluster points, where
$\sigma$ is non-trivial.
\end{conjecture}

\section{ Labeling of combinatorial AR quivers} \label{Sec: Labeling}

 In this section, we discuss about finding labels of combinatorial AR quivers.
For $A_n$ type, there are more efficiency way to find the label of each vertex in $\Gamma_Q$ than direct computations.
Similarly, for the labeling in $\Upsilon_{[\redex]}$ of other finite types, there exists analogues way to avoid large amount of computations (see Remark \ref{rmk: note} (1)).
We first discuss combinatorial AR quivers of type $A$ and generalize the argument to classical finite types.  

Let $Q$ be an AR quiver of type $A_n$.
Recall that the subquiver $B(Q)$ of the repetition quiver $\Z Q$
induces a coordinate of  the AR quiver $\Gamma_Q$. We denote by
$(\alpha)_C$ for $\alpha \in \Phi^+$ the coordinate corresponding to
$\alpha.$ On the other hand, we denote $\alpha\in \Phi^+$ by $(a,b)_\Phi$
when
$(\alpha)_C=(a,b).$

\begin{lemma} \label{lemma: position simple}\cite{B99,KKK13b}
We call the vertex $k$ in the Dynkin quiver $Q$  a {\it left
intermediate} if $Q$ has the subquiver {\rm $\xymatrix@R=3ex{ *{
\circ }<3pt> \ar@{->}[r]_<{k-1}  &*{\circ}<3pt>
\ar@{->}[r]_<{k} &*{\circ}\ar@{-}[l]^<{\ \ k+1}} $} and call
the vertex $k$ in the Dynkin quiver $Q$  a {\it right
intermediate} if $Q$ has the subquiver {\rm $\xymatrix@R=3ex{ *{
\circ }<3pt> \ar@{<-}[r]_<{k-1}  &*{\circ}<3pt>
\ar@{<-}[r]_<{k} &*{\circ}\ar@{-}[l]^<{\ \ k+1}} $.} Then we have the following properties.
\begin{enumerate}
\item[{\rm (1)}] For a simple root $\alpha_k$, we have
\begin{equation}
(\alpha_k)_C= \left\{ \begin{array}{ll}
(k,\xi_k), & \text{ if $k$ is a sink in $Q$, } \\
(n+1-k,\xi_k-n+1), & \text{ if $k$ is a source in $Q$, }\\
(1,\xi_k-k+1), & \text{ if $k$ is a right intermediate,}\\
(n,\xi_k-n+k), & \text{ if $k$ is a left intermediate.}
\end{array}
\right.
\end{equation}
\item[{\rm (2)}] If $\beta \to \alpha$ is an arrow in $\Gamma_Q$ for $\alpha, \beta \in \Phi^+$ then $(\beta, \alpha)=1$.
\end{enumerate}
Here $\xi$ is the height function such that $\text{max}\{\, \xi_k\,
|\, k=1, \cdots, n\, \}=0$.
\end{lemma}

After all, the following theorem shows how to find vertices in
$\Gamma_Q$ associated to a (non-simple) positive root in an
efficient way.  In order to introduce such methods, we distinguish types of sectional paths in AR quivers.

\begin{definition} (cf. \cite[Definition 3.3]{Oh15E})
In an AR quiver $\Gamma_Q$, a sectional path is called {\it N-sectional} if the path is upwards. On the other hand, if a sectional path is downwards, it is said to be an {\it S-sectional} path.
\end{definition}

\begin{theorem} \label{Thm:adaptedAR_comb} \cite{Oh14A}
For a positive root $\alpha=\sum_{j=k_1}^{k_2} \alpha_j$ of type $A_n$, let us call $\alpha_{k_1}$ by the {\it left end} and $\alpha_{k_2}$ by the {\it right end} of $\alpha$.
\begin{enumerate}
\item[{\rm (a)}] Every vertex in an N-sectional path in $\Gamma_Q$ shares its left end.
\item[{\rm (b)}] Every vertex in an S-sectional path in $\Gamma_Q$ shares its right end.
\end{enumerate}
\end{theorem}

Now we know how to draw the AR quiver $\Gamma_Q$ associated to the
Dynkin quiver $Q$ of $A_n$ purely combinatorially. We summarize the
procedure with the example below.

 \begin{example} \label{Ex:adapted AR quiver A}
For $ Q =\xymatrix@C=5ex{ *{ \circ }<3pt> \ar@{->}[r]_<{1}
&*{\circ}<3pt> \ar@{->}[r]_<{2} &*{\circ}<3pt>
\ar@{<-}[r]_<{3} &*{\circ}<3pt> \ar@{->}[r]_<{4} &
*{\circ}<3pt>\ar@{<-}[r]_<{ \ 5} & *{\circ}<3pt>
\ar@{-}[l]^<{\ \ \ \ 6} }$ of type $A_6$, Lemma \ref{lemma: position
simple} tells that $\Gamma_Q$ can be drawn with partial labels:
$$\scalebox{0.84}
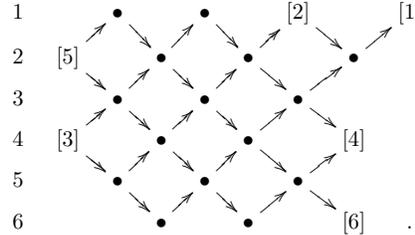
{\xymatrix@C=2ex@R=1ex{
1& &\bullet \ar@{->}[dr] & & \bullet \ar@{->}[dr] & &[2] \ar@{->}[dr] & &[1]  \\
2&[5]\ar@{->}[ur]\ar@{->}[dr]&&\bullet\ar@{->}[ur]\ar@{->}[dr]&&\bullet \ar@{->}[ur]\ar@{->}[dr]&& \bullet \ar@{->}[ur]&\\
3&& \bullet \ar@{->}[ur]\ar@{->}[dr] &&\bullet\ar@{->}[ur]\ar@{->}[dr] && \bullet\ar@{->}[ur]\ar@{->}[dr]&&\\
4&[3]\ar@{->}[ur]\ar@{->}[dr]&&\bullet\ar@{->}[ur]\ar@{->}[dr]& &\bullet \ar@{->}[ur]\ar@{->}[dr]&&[4]&\\
5&& \bullet \ar@{->}[ur]\ar@{->}[dr] && \bullet\ar@{->}[ur]\ar@{->}[dr] && \bullet \ar@{->}[ur]\ar@{->}[dr]&&\\
6&&& \bullet \ar@{->}[ur] && \bullet\ar@{->}[ur] &&[6]&.\\
}}
$$
Finally, using Theorem \ref{Thm:adaptedAR_comb}, we can {\it
complete} whole labels of $\Gamma_Q$
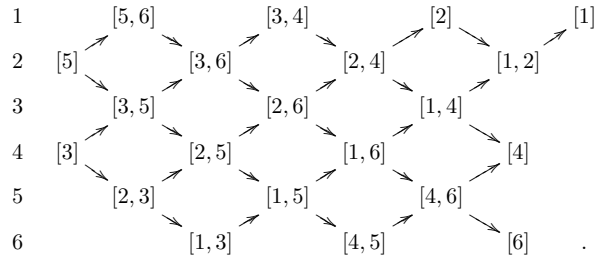
$$
\scalebox{0.84}{\xymatrix@C=2ex@R=1ex{
1& &[5,6] \ar@{->}[dr] & & [3,4] \ar@{->}[dr] & &[2] \ar@{->}[dr] & &[1]  \\
2&[5]\ar@{->}[ur]\ar@{->}[dr]&&[3,6]\ar@{->}[ur]\ar@{->}[dr]&&[2,4] \ar@{->}[ur]\ar@{->}[dr]&&[1,2] \ar@{->}[ur]&\\
3&& [3,5] \ar@{->}[ur]\ar@{->}[dr] &&[2,6]\ar@{->}[ur]\ar@{->}[dr] &&[1,4]\ar@{->}[ur]\ar@{->}[dr]&&\\
4&[3]\ar@{->}[ur]\ar@{->}[dr]&&[2,5]\ar@{->}[ur]\ar@{->}[dr]& &[1,6] \ar@{->}[ur]\ar@{->}[dr]&&[4]&\\
5&& [2,3]\ar@{->}[ur]\ar@{->}[dr] && [1,5]\ar@{->}[ur]\ar@{->}[dr] && [4,6] \ar@{->}[ur]\ar@{->}[dr]&&\\
6&&& [1,3] \ar@{->}[ur] && [4,5]\ar@{->}[ur] &&[6]&.\\
}}
$$
\end{example}

Now, we generalize the above arguments in $\Gamma_Q$. In order to find analogous results in $\Upsilon_{[\redex]}$ of every finite type, we introduce the notion of {\it component}:

\begin{definition} \label{Def:comp}
 Let  $\alpha= \sum_{i\in J} c_i \epsilon_i$ and $\beta= \sum_{i\in J} d_i \epsilon_i$. (Note that $J$ need not to be the same as $I$.)
\begin{enumerate}
\item If $i\in I$ satisfies $c_i \neq 0$ then $\epsilon_i$ is called a component of $\alpha$.
\item  If $i\in I$ satisfies $c_i>0$ (resp. $c_i<0$) then $\epsilon_i$ is called a {\it positive component } (resp. {\it negative component} of $\alpha$.
\item
We say $\alpha$ and $ \beta$ share a component if there is $i\in I$ such that $\epsilon_i$ is a positive component to both $\alpha$ and $\beta$ or a negative component to both $\alpha$ and $\beta$.
\end{enumerate}
\end{definition}

\begin{remark}
In $A_n$ type, we have $[i,j]=\epsilon_i-\epsilon_{j+1}.$ Hence Theorem \ref{Thm:adaptedAR_comb} can be restated as follows:
An $N$-sectional (resp. $S$-sectional) path in $\Gamma_Q$ shares a positive (resp. negative) component. In short, each sectional path in $\Gamma_Q$ shares a component.
\end{remark}

For type $A_n$, recall that the action $s_i$  on $\Phi^+$ can be described as follows:
\begin{align} \label{Eqn:s_i action}
[j,k] \mapsto\begin{cases}
[j,k-1] & \text{ if }j<k=i,\\
[j+1,k]   & \text{ if } j=i<k,\\
[j,k+1]    & \text{ if } j<k=i-1,\\
[j-1,k]   & \text{ if } j=i+1<k,\\
-[i] & \text{ if } i=j=k,\\
[j,k] & \text{ otherwise.}
\end{cases}
\end{align}


Then the following lemma is an easy consequence induced from the
action of simple reflection on $\Phi^+$:

\begin{lemma} \label{Lem:AR_comb_prop1}
Let $s_t$ be a simple reflection on $\W$ of type $A_n$ and $[i,j]
\seteq \sum_{k=i}^j \alpha_k$ for $i,j\in I$.
\begin{enumerate}
\item[{\rm (1)}]  If $s_t[i,k]$, $s_t [j,k]\in \Phi^+$  then $s_{t}[i,k]=[i',k']$ and $s_{t}[j,k]=[j',k']$ for some $i',j'\leq k' \in \{1,2, \cdots, n\}$.
\item[{\rm (2)}] If $s_t[i,j]$, $s_t[i,k]\in \Phi^+$  then $s_{t}[i,j]=[i',j']$ and $s_{t}=[i',k']$ for some $i'\leq j', k' \in \{1,2, \cdots, n\}$.
\end{enumerate}
\end{lemma}

\begin{proposition} \label{Prop:AR_label}
Let $\redex=(s_{i_1},s_{i_2},\cdots, s_{i_N})$ be a reduced
expression of $w\in \W$ of type $A_n$ and $\Gamma_{[\redex]}$ be the
combinatorial AR quiver.
\begin{enumerate}
\item[{\rm (a)}]  If there is an arrow from $\beta^\redex_{k_1}$ in the $l$-th residue to $\beta^\redex_{k_2}$ in the $(l-1)$-th residue,
then the corresponding positive roots $[i_1,j_1]$ and $[i_2,j_2]$ to
$\beta^\redex_{k_1}$ and $\beta^\redex_{k_2}$ satisfy $i_1=i_2$.
\item[{\rm (b)}]  If there is an arrow from $\beta^\redex_{k_1}$ in the $l$-th residue to $\beta^\redex_{k_2}$ in the $(l+1)$-th residue,
then the corresponding positive roots $[i_1,j_1]$ and $[i_2,j_2]$ to
$\beta^\redex_{k_1}$ and $\beta^\redex_{k_2}$ satisfy $j_1=j_2$.
\end{enumerate}
\end{proposition}
\begin{proof}
{\rm (a)} The arrow from $\beta^\redex_{k_1}$ in the $l$-th residue
to $\beta^\redex_{k_2}$ on the $(l-1)$-th residue implies that
 $k_1 > k_2$ and
\begin{eqnarray} &&
  \parbox{95ex}{
the vertices $\{\beta^\redex_k \, | \, k=k_2+1, \cdots, k_1-1\}$ in
$\Upsilon_{[\redex]}$ are not on the $l$-th or $(l-1)$-th residue.
}\label{eq: observation}
\end{eqnarray}

Denote $\redex_{\le k_2-1}=s_{i_1} s_{i_2} \cdots s_{k_2-1}$. Then
 $[i_1, j_1]=  \redex_{\le k_2-1} s_{i_{k_2}} s_{i_{k_2}+1} \cdots s_{i_{k_1-1}}(\alpha_{i_{k_1}}=[l])$ and
 $[i_2, j_2]=  \redex_{\le k_2-1} (\alpha _{i_{k_2}}=[l-1])$. Using \eqref{Eqn:s_i action} and \eqref{eq: observation}, we have
$$s_{i_{k_2}} s_{i_{k_2}+1} \cdots s_{i_{k_1-1}}(\alpha_{i_{k_1}})=[l-1,j]$$ for some $j\geq l$.
Then the first assertion follows from Lemma \ref{Lem:AR_comb_prop1}. \\
{\rm (b)}  The same argument as that in the proof of {\rm (a)}
works.
\end{proof}

\begin{theorem} \label{Thm: labeling of type A} For any $\Upsilon_{[\redex]}$ of type $A$,
if two roots $\alpha$ and $\beta$ are in an N-sectional $($resp. S-sectional$)$ path then $\alpha$ and $\beta$ share  their positive $($resp. negative$)$ components.
\end{theorem}

\begin{example}
Let $\redez=(s_1, s_2, s_1, s_3, s_5, s_4, s_3, s_2, s_3, s_5, s_4,
s_1, s_3, s_2, s_3)$ of $A_5$. We can easily find that labels of
sinks and sources of the quiver $\Upsilon_{[\redez]}$ are $[1]$,
$[5]$ and $[3]$. In addition, one can compute that
$\beta^\redez_3=\alpha_2$ and
$\beta^\redez_{13}=-w_0(s_3s_2(\alpha_3))=\alpha_4$ easily. Hence
$\Gamma_{[\redex]}$ has the form
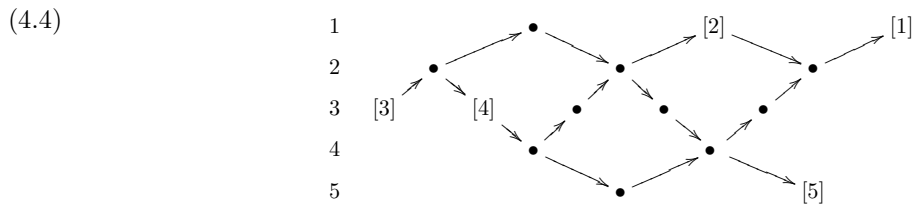
\begin{equation}
\scalebox{0.84}{\xymatrix@C=2ex@R=1ex{
1&&&& \bullet \ar@{->}[drr] &&  && [2]\ar@{->}[drr] && && [1] \\
2&& \bullet \ar@{->}[dr]\ar@{->}[urr] &&  && \bullet \ar@{->}[dr]\ar@{->}[urr] &&  && \bullet \ar@{->}[urr]\\
3& [3] \ar@{->}[ur] && [4] \ar@{->}[dr] && \bullet \ar@{->}[ur] && \bullet \ar@{->}[dr] && \bullet\ar@{->}[ur] \\
4&&  && \bullet \ar@{->}[ur]\ar@{->}[drr]  &&&& \bullet\ar@{->}[ur]\ar@{->}[drr]\ \\
5&&&&  && \bullet\ar@{->}[urr] &&  && [5] \\
}}
\end{equation}
 By Proposition \ref{Prop:AR_label}, we can find almost all labels of $\Upsilon_{[\redez]}$ as follows:
 \begin{equation}
\scalebox{0.84}
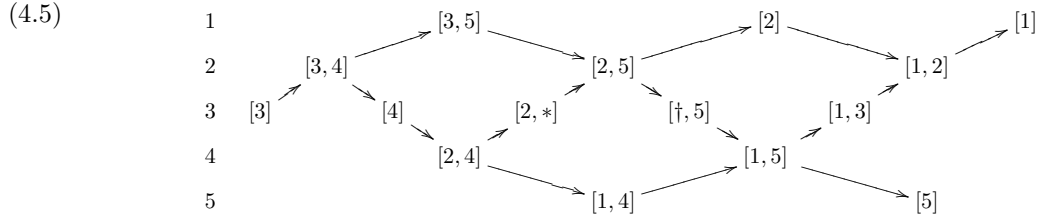
{\xymatrix@C=2ex@R=1ex{
1&&&& [3,5] \ar@{->}[drr] &&  && [2]\ar@{->}[drr] && && [1] \\
2&& [3,4] \ar@{->}[dr]\ar@{->}[urr] &&  && [2,5] \ar@{->}[dr]\ar@{->}[urr] &&  && [1,2]\ar@{->}[urr]\\
3& [3] \ar@{->}[ur] && [4] \ar@{->}[dr] && [2,*] \ar@{->}[ur] && [\dagger,5] \ar@{->}[dr] && [1,3]\ar@{->}[ur] \\
4&&  && [2,4] \ar@{->}[ur]\ar@{->}[drr]  &&&& [1,5]\ar@{->}[ur]\ar@{->}[drr]\ \\
5&&&&  && [1,4]\ar@{->}[urr] &&  && [5] \\
}}
\end{equation}
Finally, we can conclude that $*=3$ and $\dagger=4$, since the
labels of $\Upsilon_{[\redez]}$ coincide with $\Phi^+$.
\end{example}



By applying similar arguments of  Lemma \ref{Lem:AR_comb_prop1} and Proposition \ref{Prop:AR_label}, we have the following theorem for classical finite types ABCD:
\begin{theorem} \label{Rem:label_Dtype}
For any $\Upsilon_{[\redex]}$ of classical finite types, a sectional path shares a component; that is,
$$ \text{ if two roots $\alpha$ and $\beta$ are in a sectional path then $\alpha$ and $\beta$ share one component. }$$
\end{theorem}

We do not know whether the above theorem holds for exceptional types $E$ and $F_4$ or not. However, we can observe the following proposition without consideration of types: 

\begin{proposition} \label{prop: sectional simplicity}
 For $\al$ and $\beta$ in a sectional path in $\Upsilon_{[\redex]}$, there exists no $\{ \gamma_{i} \ | \ 1 \le i \le r \}$ in the same sectional path such that
$$ \sum_{i=1}^r \gamma_i =\alpha+\beta \quad \text{ and } \quad  \gamma_i \ne \al,\be \ \ \text{ for all } 1 \le i \le r.$$
\end{proposition}

\begin{proof}
Our assertion for classical types follows from the previous theorem. By considering a sectional path
$$\left\{ w(\al_{i_1}) \to ws_{i_1}(\al_{i_2}) \to \cdots \to w\prod_{s=1}^{k} s_{i_{s}}(\al_{i_{k+1}}) \right\}$$
for any $w$ of finite type, one can check our assertion in general.
\end{proof}

\begin{example} Recall that the set of positive roots can be expressed as $$\{ \, \epsilon_i \pm  \epsilon_{j} \,  | \,  1\leq i<j \leq n \, \}.$$
For type $D_5$, consider the reduced expression
\[ \redez= (s_2, s_1, s_3, s_2, s_1, s_5, s_3 , s_2, s_1, s_4, s_3, s_2, s_1, s_5, s_3, s_2, s_1, s_4, s_3, s_5). \]

The combinatorial AR quiver $\Upsilon_{[\redez]}$ has the form of
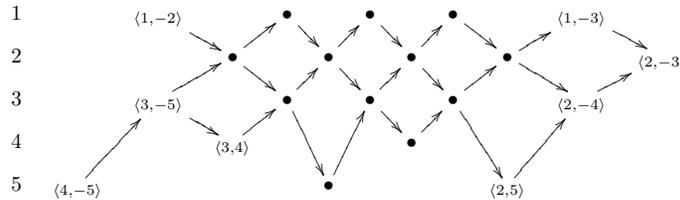
\[  \scalebox{0.8}{\xymatrix@C=2ex@R=1ex{
1&&\srt{1}{-2}\ar@{->}[dr]  && \bullet \ar@{->}[dr]  &&  \bullet \ar@{->}[dr]  &&\bullet \ar@{->}[dr] && \srt{1}{-3} \ar@{->}[dr]\\
2&&& \bullet \ar@{->}[dr]\ar@{->}[ur]  &&\bullet \ar@{->}[dr]\ar@{->}[ur]  &&\bullet \ar@{->}[dr]\ar@{->}[ur]  && \bullet \ar@{->}[dr]\ar@{->}[ur] && \srt{2}{-3} \\
3& &\srt{3}{-5}\ar@{->}[dr]\ar@{->}[ur]  &&\bullet\ar@{->}[ddr]\ar@{->}[ur]  &&\bullet\ar@{->}[dr]\ar@{->}[ur]  && \bullet \ar@{->}[ddr]\ar@{->}[ur] && \srt{2}{-4}\ar@{->}[ur] \\
4&&& \srt{3}{4}\ar@{->}[ur]  &&&& \bullet \ar@{->}[ur] \\
5&\srt{4}{-5}\ar@{->}[uur]  &&&& \bullet \ar@{->}[uur]  &&&& \srt{2}{5}\ar@{->}[uur]
}}.\]
Here $\epsilon_i \pm \epsilon_j$ is denoted by $\left< i, \pm j \right>$. Note that the labels filled in the previous quiver are not hard to find by direct computations. Now, by Theorem \ref{Rem:label_Dtype}, we can complete to find all labels in $\Upsilon_{[\redez]}$.
\[  \scalebox{0.8}
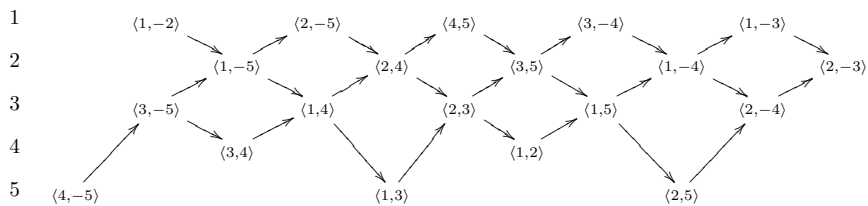
{\xymatrix@C=2ex@R=1ex{
1&&\srt{1}{-2}\ar@{->}[dr]  && \srt{2}{-5}\ar@{->}[dr]  &&  \srt{4}{5}\ar@{->}[dr]  && \srt{3}{-4}\ar@{->}[dr] && \srt{1}{-3} \ar@{->}[dr]\\
2&&& \srt{1}{-5}\ar@{->}[dr]\ar@{->}[ur]  &&\srt{2}{4}\ar@{->}[dr]\ar@{->}[ur]  && \srt{3}{5}\ar@{->}[dr]\ar@{->}[ur]  && \srt{1}{-4}\ar@{->}[dr]\ar@{->}[ur] && \srt{2}{-3} \\
3& &\srt{3}{-5}\ar@{->}[dr]\ar@{->}[ur]  &&\srt{1}{4}\ar@{->}[ddr]\ar@{->}[ur]  &&\srt{2}{3}\ar@{->}[dr]\ar@{->}[ur]  && \srt{1}{5}\ar@{->}[ddr]\ar@{->}[ur] && \srt{2}{-4}\ar@{->}[ur] \\
4&&& \srt{3}{4}\ar@{->}[ur]  &&&& \srt{1}{2}\ar@{->}[ur] \\
5&\srt{4}{-5}\ar@{->}[uur]  &&&& \srt{1}{3}\ar@{->}[uur]  &&&& \srt{2}{5}\ar@{->}[uur]
}}\]
\end{example}

\begin{example} In Example \ref{Ex: C3}, $\Upsilon_{[\redez]}$ can be also labeled in terms of orthonormal basis:
$$\Upsilon_{[\redez]}=
 \scalebox{0.84}{\xymatrix@C=1ex@R=1ex{
1& \epsilon_{1}-\epsilon_2\ar@{->}[dr] &&&&\epsilon_1 +\epsilon_2 \ar@{->}[dr]&&&&\\
2&& \epsilon_1-\epsilon_3 \ar@{=>}[dr] && \epsilon_1+\epsilon_3\ar@{->}[ur] && \epsilon_2-\epsilon_3 \ar@{=>}[dr] && \epsilon_2+\epsilon_3 \ar@{=>}[dr]& \\
3&&& 2\epsilon_1 \ar@{=>}[ur] &&&&  2\epsilon_2 \ar@{=>}[ur] && 2\epsilon_3 }},
$$
which implies Theorem \ref{Rem:label_Dtype}. Note that, for any reduced expression,
every positive root of the form $2\epsilon_i$ has residue $n$ and any positive root has residue $n$ is of the form $2\epsilon_i$.
\end{example}

\section{Application to KLR algebras  and PBW bases} \label{sec: KLR}

 In this section, we apply our results in previous sections to the representation theory of KLR algebras which were introduced by
 Khovanov-Lauda \cite{KL09} and Rouquier \cite{R08}, independently.

\subsection{KLR algebra} \label{subsec: KLR}
Let $I$ be an index set. A \emph{symmetrizable Cartan datum}
$\mathsf{D}$ is a quintuple $(\cmA,\wl,\Pi,\wl^{\vee},\Pi^{\vee})$
consisting of
{\rm (a)} an integer-valued matrix $\cmA=(a_{ij})_{i,j \in I}$,
called \emph{the symmetrizable generalized Cartan matrix},
{\rm (b)} a free abelian group $\wl$, called the \emph{weight lattice},
{\rm (c)} $\Pi= \{ \alpha_i \in \wl \mid \ i \in I \}$, called
the set of \emph{simple roots},
{\rm (d)} $\wl^{\vee}\seteq {\rm Hom}(\wl, \Z)$, called the \emph{coweight lattice},
{\rm (e)} $\Pi^{\vee}= \{ h_i \ | \ i \in I \}\subset P^{\vee}$, called
the set of \emph{simple coroots}, satisfying
$$\langle h_i,\alpha_j \rangle = a_{ij} \text{ for all } i,j \in I \text{ and $\Pi$ is linearly independent}.$$

The free abelian group $\rl\seteq\soplus_{i \in I} \Z \alpha_i$ is
called the \emph{root lattice}. Set $\rl^{+}= \sum_{i \in I}
\Z_{\ge 0} \alpha_i$. 

Let $\cor$ be a commutative ring. For $i,j\in I$ such that $i\not=j$
and let us take  a family of polynomials $(Q_{ij})_{i,j\in I}$ in
$\cor[u,v]$ which are of the form
\begin{equation} \label{eq:Q}
Q_{ij}(u,v) =
\delta(i \ne j)\sum\limits_{ \substack{ (p,q)\in \Z^2_{\ge0} \\
d_i \times  p+d_j \times q=-d_i \times a_{ij}}} t_{i,j;p,q} u^p v^q
\end{equation}
with $t_{i,j;p,q}\in\cor$, $t_{i,j;p,q}=t_{j,i;q,p}$ and
$t_{i,j;-a_{ij},0} \in \cor^{\times}$. Thus we have
$Q_{i,j}(u,v)=Q_{j,i}(v,u)$.

We denote by $\mathfrak{S}_{n} = \langle \mathfrak{s}_1, \ldots,
\mathfrak{s}_{n-1} \rangle$ the symmetric group on $n$ letters,
where $\mathfrak{s}_i\seteq (i, i+1)$ is the transposition of $i$
and $i+1$. Then $\mathfrak{S}_n$ acts on $I^n$ by place
permutations.

For $n \in \Z_{\ge 0}$ and $\beta \in \rl^+$ such that $\het(\be) =
n$, we set
$$I^{\beta} = \{\nu = (\nu_1, \ldots, \nu_n) \in I^{n} \ | \ \alpha_{\nu_1} + \cdots + \alpha_{\nu_n} = \beta \}.$$

\begin{definition}
For $\beta \in \rl^+$ with $|\beta|=n$, the {\em
Khovanov-Lauda-Rouquier algebra}  $R(\beta)$  at $\beta$ associated
with a symmetrizable Cartan datum $(\cmA,\wl,
\Pi,\wl^{\vee},\Pi^{\vee})$ and a matrix $(Q_{ij})_{i,j \in I}$ is
the $\Z$-gradable $\cor$-algebra generated by the elements $\{
e(\nu) \}_{\nu \in  I^{\beta}}$, $ \{x_k \}_{1 \le k \le n}$, $\{
\tau_m \}_{1 \le m \le n-1}$ satisfying the following defining
relations:
\begin{align*}
& e(\nu) e(\nu') = \delta_{\nu, \nu'} e(\nu), \ \
\sum_{\nu \in  I^{\beta} } e(\nu) = 1, \ \ x_{k} x_{m} = x_{m} x_{k}, \ \ x_{k} e(\nu) = e(\nu) x_{k}, \allowdisplaybreaks \\
& \tau_{m} e(\nu) = e(\mathfrak{s}_{m}(\nu)) \tau_{m}, \ \
\tau_{k}\tau_{m} = \tau_{m} \tau_{k} \ \ \text{if} \ |k-m|>1,  \ \  \tau_{k}^2 e(\nu) = Q_{\nu_{k}, \nu_{k+1}} (x_{k}, x_{k+1}) e(\nu), \allowdisplaybreaks \\
& (\tau_{k} x_{m} - x_{\mathfrak{s}_k(m)} \tau_{k}) e(\nu) =
\begin{cases}
-e(\nu) \ \ & \text{if} \ m=k, \nu_{k} = \nu_{k+1}, \\
e(\nu) \ \ & \text{if} \ m=k+1, \nu_{k}=\nu_{k+1}, \\
0 \ \ & \text{otherwise},
\end{cases}  \allowdisplaybreaks \\
& (\tau_{k+1} \tau_{k} \tau_{k+1}-\tau_{k} \tau_{k+1} \tau_{k}) e(\nu)  =\begin{cases} \dfrac{Q_{\nu_{k}, \nu_{k+1}}(x_{k}, x_{k+1}) -
Q_{\nu_{k}, \nu_{k+1}}(x_{k+2}, x_{k+1})} {x_{k} - x_{k+2}}e(\nu) \
\ & \text{if} \
\nu_{k} = \nu_{k+2}, \\
0 \ \ & \text{otherwise}.
\end{cases}
\end{align*}
\end{definition}

For $\beta, \gamma \in \rl^+$ with $\het(\beta)=m$, $\het(\gamma)=
n$,
 set
$$e(\beta,\gamma)=\displaystyle\sum_{\substack%
{\nu \in I^{m+n}, \\ (\nu_1, \ldots ,\nu_m) \in I^{\beta},\\
(\nu_{m+1}, \ldots ,\nu_{m+n}) \in I^{\gamma}}} e(\nu) \in
R(\beta+\gamma).$$ Then $e(\beta,\gamma)$ is an idempotent. Let
\begin{align} \label{eq:embedding}
R( \beta)\otimes R( \gamma  )\to e(\beta,\gamma)R(
\beta+\gamma)e(\beta,\gamma)
\end{align} be the $\cor$-algebra homomorphism given by
\begin{equation*}
\begin{aligned}
& e(\mu)\otimes e(\nu)\mapsto e(\mu*\nu) \ \ (\mu\in I^{\beta}),\\
& x_k\otimes 1\mapsto x_ke(\beta,\gamma) \ \  (1\le k\le m),  \quad 1\otimes x_k\mapsto x_{m+k}e(\beta,\gamma) \ \  (1\le k\le n), \\
& \tau_k\otimes 1\mapsto \tau_ke(\beta,\gamma) \ \  (1\le k<m),
\quad \ 1\otimes \tau_k\mapsto \tau_{m+k}e(\beta,\gamma) \ \  (1\le
k<n),
\end{aligned}
\end{equation*}
where  $\mu*\nu$ is the concatenation of $\mu$ and $\nu$; i.e.,
$\mu*\nu=(\mu_1,\ldots,\mu_m,\nu_1,\ldots,\nu_n)$.

\medskip
For a  $R(\beta)$-module $M$ and a  $R(\gamma)$-module $N$, we
define the \emph{convolution product} $M\conv N$ by
$$M\conv N \seteq R(\beta + \gamma) e(\beta,\gamma)
\otimes_{R(\beta )\otimes R( \gamma)}(M\otimes N). $$

For a graded $R(\be)$-module $M=\bigoplus_{k \in \Z} M_k$, we define
$qM =\bigoplus_{k \in \Z} (qM)_k$, where
 \begin{align*}
 (qM)_k = M_{k-1} & \ (k \in \Z).
 \end{align*}
We call $q$ the \emph{grading shift functor} on the category of
graded $R(\be)$-modules.

Let $\Rep(R(\be))$ be the category consisting of finite dimensional
graded $R(\be)$-modules and
 $[\Rep(R(\be))]$ be the Grothendieck group of $\Rep(R(\be))$. Then $[\Rep(R)] \seteq \soplus_{\be \in \rl^+} [\Rep(R(\be))]$
 has a natural $\Z[q,q^{-1}]$-algebra structure induced by the convolution product $\conv$ and the grading shift functor $q$.
In this paper, we usually ignore grading shifts.

For an $R(\be)$-module $M$ and an $R(\gamma_k)$-module $M_k$ $(1 \le
k \le n)$, we denote by
$$ M^{\conv 0} \seteq \cor, \quad M^{\conv r} = \underbrace{ M \conv \cdots \conv M }_r, \quad \dct{k=1}{n} M_k = M_1 \conv \cdots \conv M_n.$$

\begin{theorem}[{\cite{KL09, R08}}] \label{Thm:categorification}
For a given symmetrizable Cartan datum $\mathsf{D}$, let
$U_{\Z[q,q^{-1}]}(\g)^\vee$ the dual of the integral form of the
negative part of the quantum group $U_q(\g)$ associated with
$\mathsf{D}$ and $R$ be the KLR algebra associated with $\mathsf{D}$
and $(Q_{ij}(u,v))_{i,j \in I}$. Then we have
\begin{align}
U^-_{\Z[q,q^{-1}]}(\g)^{\vee} \simeq [\Rep(R)]. \label{eq:KLRU}
\end{align}
\end{theorem}

\medskip

From now on, we shall deal with the representation theory of KLR
algebras which are associated to the Cartan matrix $\cmA$ of finite
types.

\begin{definition} \label{Def: minimal pair}\cite[\S 2.1]{Mc12}.
For a convex total order $<$ on $\Phi(w)$, a pair $(\alpha,\beta)$ with
$\alpha<\beta$ is called a {\it minimal pair of $\gamma \in \Phi(w)$
with respect to the  convex total order $<$} if
\begin{itemize}
\item $\gamma=\alpha+\beta \in \Phi(w)$,
\item there exist {\it no} pair $(\alpha',\beta') \in (\Phi(w))^2$ such that $\gamma=\alpha'+\beta' \text{ and }\alpha<\alpha'<\gamma<\beta'<\beta.$
\end{itemize}
\end{definition}

\begin{convention} \label{conv: conv1}
For a reduced expression $\redex$ of $w \in \W$, we fix a labeling
of $\Phi(w)$ as $\{\beta^{\tw}_k \ | \ 1 \le k \le \ell(w) \}$.
\begin{enumerate}
\item[{\rm (i)}] We identify a sequence $\um_\tw=(m_1,m_2,\ldots,m_{\ell(w)}) \in \Z_{\ge 0}^{\ell(w)}$ with
$$(m_1\beta^\tw_1 ,m_2\beta^\tw_2,\ldots,m_{\ell(w)} \beta^\tw_{\ell(w)}) \in (\rl^+)^{\ell(w)}.$$
\item[{\rm (ii)}] For a sequence $\um_\tw$ and another reduced expression $\tw'$ of $w$, $\um_{\tw'}$ is a sequence in $\Z_{\ge 0}^{\ell(w)}$
by considering $\um_\tw$ as a sequence of positive roots,
rearranging with respect to $<_{\tw'}$ and applying the convention
{\rm (i)}.
\item[{\rm (iii)}]
For a sequence $\um_\tw\in \Z_{\ge 0}^{\ell(w)}$, the weight
$\wt(\um_\tw)$ of $\um_\tw$ is defined by $\sum_{i=1}^{\ell(w)}
m_i\beta^\tw_i \in \rl^+$.
\end{enumerate}
\end{convention}

We usually drop the script $\tw$ if there is no fear of confusion.

\begin{definition} [\cite{Mc12,Oh15E}] \label{def: redezletb}
For sequences $\um$, $\um' \in \Z_{\ge 0}^{\ell(w)}$, we define an
order $\le^\tb_{\tw}$ as follows:
\begin{eqnarray*}&&
\parbox{95ex}{
$\um'=(m'_1,\ldots,m'_{\ell(w)}) <^\tb_{\redex}
\um=(m_1,\ldots,m_{\ell(w)}) $ if and only if  $\wt(\um)=\wt(\um')$ and there exist integers
$k$, $s$ such that $1 \le k \le s \le \ell(w)$, $m_t'=m_t$ $(t<k)$,
$m'_k  <  m_k$, and $m_t'=m_t$ $(s<t\le \ell(w))$, $m'_{s}  <
m_{s}$. }
\end{eqnarray*}
\end{definition}

The following order on sequences of positive roots was introduced in
\cite{Oh15E}.

\begin{definition} \cite{Oh15E} \label{def: redezprectb}
For sequences $\um$, $\um' \in \Z_{\ge 0}^{\ell(w)}$, we define an
order $\prec^\tb_{[\tw]}$ as follows:
\begin{eqnarray}&&
\parbox{65ex}{
$\um'=(m'_1,\ldots,m'_{\ell(w)}) \prec^\tb_{[\tw]}
\um=(m_1,\ldots,m_{\ell(w)}) $ if and only if  $\um'_{\tw'}
<^\tb_{\tw'}  \um_{\tw'}$ for all reduced expression $\tw' \in
[\tw]$. }\label{eq: redezprectb}
\end{eqnarray}
\end{definition}
Note that  $\prec^\tb_{[\tw]}$ is {\it far coarser} than
$<^\tb_{\tw}$.

\begin{definition} \label{def: [tw]-simple pair}
A pair $(\alpha,\beta)$ of positive roots is {\it $[\redez]$-simple} if there exists no sequence $\um \in \Z^{N}_{\ge 0}$ such that
\begin{align}\label{eq: [tw]-simple pair}
\um \prec^\tb_{[\redez]}  (\alpha,\beta).
\end{align}
\end{definition}

\medskip

For a module $M$, we denote by $\hd(M)$ the head of $M$ and by
$\soc(M)$ the socle of $M$.

\begin{theorem} \cite{BKM12,Mc12}\label{thm: BkMc}
Let $R$ be the KLR algebra corresponding to a Cartan matrix $\cmA$
of finite type. For each positive root $\beta \in \PR$, there exists
a simple module $S_{\redez}(\beta)$ satisfying the following
properties:
\begin{itemize}
\item[({\rm a})] $S_{\widetilde{w}_0}(\beta)^{\conv m}$ is a simple $R(m\be)$-module.
\item[({\rm b})] For every $\um_{\redez} \in \Z_{\ge 0}^{\ell(w_0)}$, there exists a non-zero $R$-module homomorphism
\begin{equation} \label{eq: def of r}
\begin{aligned}
&\rmat{\um} \colon \Stom \seteq S_{\redez}(\beta_1)^{\conv m_1}\conv\cdots \conv S_{\redez}(\beta_{\ell(w_0)})^{\conv m_{\ell(w_0)}}\\
&\hspace{20ex}{\Lto}\Sgetsm \seteq
S_{\redez}(\beta_{\ell(w_0)})^{\conv m_{\ell(w_0)}}\conv\cdots \conv
S_{\redez}(\beta_1)^{\conv m_1}.
\end{aligned}
\end{equation}
such that $${\rm Hom}_{R(\wt(\um))}(\Stom,\Sgetsm)=\cor \cdot
\rmat{\um}$$ and $${\rm Im}(\rmat{\um}) \simeq \hd\left(\Stom\right)
\simeq \soc\left(\Sgetsm\right) \text{ is simple.}$$
\item[({\rm c})] For any sequence $\um_{\redez} \in \Z_{\ge 0}^{\ell(w_0)}$, we have
\begin{align} \label{eq: not good filter}
[\Stom] \in [{\rm Im}(\rmat{\um})] + \displaystyle\sum_{
\substack{\um' <^{\tb}_\redez \um \\ \wt(\um')=\wt(\um)} } \Z_{\ge
0}[q^{\pm1}] [{\rm Im}(\rmat{\um'})].
\end{align}
\item[({\rm d})] For any sequence $\um_{\redez} \in \Z_{\ge 0}^{\ell(w_0)}$, $\Stom$ has
a unique simple head $\hd\left(\Stom\right)$ and $\Stom \not \simeq
\overset{\to}{S}_{\redez}(\um')$ if $\um \ne \um'$.
\item[({\rm e})] For every simple $R$-module $M$, there exists a unique sequence $\um \in \Z_{\ge 0}^{\mathsf{N}}$ such that
$M \simeq {\rm Im}(\rmat{\um}) \simeq {\rm hd}\big(\Stom \big).$
\item[({\rm f})] For any minimal pair $(\beta^\redez_k,\beta^\redez_l)$ of $\beta^\redez_j=\beta^\redez_k+\beta^\redez_l$ with respect to $<_\redez$,
there exists an exact sequence $$ \qquad 0 \to S_{\redez}(\beta_j)
\to S_{\redez}(\beta_k) \conv S_{\redez}(\beta_l) \overset{{\mathbf
r}_{\um}}{\Lto} S_{\redez}(\beta_l) \conv S_{\redez}(\beta_k) \to
S_{\redez}(\beta_j) \to 0,$$ where $\um_{\redez} \in \Z_{\ge
0}^{\ell(w_0)}$ such that $m_k=m_l=1$ and $m_i=0$ for all $i \ne
k,l$.
\end{itemize}
\end{theorem}

Note that the set ${\rm Irr}(R)$ of isomorphism classes of all
simple $R$-modules
 forms a natural basis of $[\Rep(R)]$ and does {\it not}
depend on the choice of reduced expression $\redez$ of $w_0$.

\medskip

We also note that Theorem \ref{thm: BkMc} implies that
\begin{itemize}
\item[{\rm (i)}] the subset $\overset{\to}{S}_\redez(R) \seteq \left\{
\big[\overset{\to}{S}_\redez(\um) \big] \ | \ \um_\redez \in \Z_{\ge
0}^{\ell(w_0)} \right\}$ of isomorphism classes of $R$-modules
 forms another basis of $[\Rep(R)]$,
\item[{\rm (ii)}] $<^\tb_\redez$ can be interpreted as a unitriangular matrix which plays the role of the transition matrix between $\overset{\to}{S}_\redez(R)$ and
${\rm Irr}(R)$ for {\it any} reduced expression $\redez$ of $w_0$.
\end{itemize}

\subsection{$\overset{\to}{S}_{[\redez]}(R)$ and $\prec^\tb_{[\redez]}$} In this subsection, we will apply the observations in the previous sections to the representation theory of KLR-algebras and
PBW-bases:

\begin{theorem} \label{thm: app KLR1} \cite[Theorem 5.13]{Oh15E} For any $\redez$ of $w_0$ and $\um_\redez \in \Z_{\ge 0}^{\ell(w_0)}$, we can define the module $\overset{\to}{S}_{[\redez]}(\um)$; i.e,
$$\overset{\to}{S}_\redez(\um_\redez)  \simeq \overset{\to}{S}_{\redez'}(\um_{\redez'}) \quad \text{ for all } \redez,\redez' \in [\redez].$$
Moreover, we can refine the transition matrix between
$\overset{\to}{S}_{[\redez]}(R)\seteq \{
\overset{\to}{S}_{[\redez]}(\um) \ | \ \um \in \Z^{\ell(w_0)}_{\ge
0} \}$ and ${\rm Irr}(R)$ by replacing $<^\tb_{\redez}$ with the far
coarser order $\prec^\tb_{[\redez]}$.
\end{theorem}

\begin{remark}
For any $\redez ,\redez' \in [\redez]$, Theorem \ref{thm: BkMc}
tells that
$$ S_{\redez}(\be) \simeq  S_{\redez'}(\be)  \quad \text{ for all } \be \in \PR.$$
Thus we denote by $S_{[\redez]}(\be)$ the simple module
$S_{\redez'}(\be)$ for any $\redez' \in [\redez]$ and $\be \in \PR$.
\end{remark}

\begin{proposition} \label{prop: incomp simple}
Let $(\al,\be)$ be an incomparable pair of positive roots with
respect to the order $\prec_{[\redez]}$. Then $(\al,\be)$ is $[\redez]$-simple and we have
$$ S_{[\redez]}(\al) \conv S_{[\redez]}(\be) \simeq S_{[\redez]}(\be) \conv S_{[\redez]}(\al) \text{ is simple.} $$
\end{proposition}

\begin{proof}
By Corollary \ref{cor: imcom adja}, there exist $\redez' \in
[\redez]$ and $k \in \Z_{\ge 1}$ such that $\al=\be^{\redez'}_k$ and
$\be=\be^{\redez'}_{k+1}$. Let us denote by $(\al,\be)$ the sequence
$\um_{\redez'}$ such that $m_k=m_{k+1}=1$ and $m_i=0$ for all $i \ne
k,k+1$. Then there is no $\um_{\redez'}$ such that $\um
<^\tb_{\redez'} (\al,\be)$. Hence Theorem \ref{thm: BkMc} {\rm (c)}
tells that the composition series of $S_{[\redez]}(\al) \conv
S_{[\redez]}(\be)$ consists of ${\rm Im}(\rmat{(\al,\be)})$. Then
our assertion follows from Theorem \ref{thm: BkMc} {\rm (b)}.
\end{proof}

\begin{remark}
Proposition \ref{prop: incomp simple} tells that $S_{[\redez]}(\al)$ and $S_{[\redez]}(\be) $ {\it commutes} up to grading shift (or {\it $q$-commutes}) if $\alpha$ and $\beta$ are incomparable with respect to $\prec_{[\redez]}$.
However, the converse is not true.
In Proposition \ref{prop: q-comm comparable} below, we will show that for comparable pair $(\alpha,\beta)$, $S_{[\redez]}(\al)$ and $S_{[\redez]}(\be)$ commutes
if they lie in the same sectional path in $\Upsilon_{[\redez]}$, which is a generalization of \cite[Proposition 4.2]{Oh15E}.
\end{remark}

\begin{proof}[Proof of Theorem \ref{thm: app KLR1}] By proposition \ref{prop: incomp simple}, the isomorphism class of
the module $\overset{\to}{S}_\redez(\um_\redez)$ and the
homomorphism $\rmat{\um_\redez}$ does not depend on the choice of
$\redez \in [\redez]$. Thus our first assertion follows. By applying
the first assertion to \eqref{eq: not good filter} for all
$\redez'\in [\redez]$, we have
\begin{align*}
[\overset{\to}{S}_{[\redez]}(\um)] \in [{\rm Im}(\rmat{\um})] +
\displaystyle\sum_{ \substack{\um' <^{\tb}_{\redez'} \um \text{ for
all } \redez' \in [\redez] \\ \wt(\um')=\wt(\um)} } \Z_{\ge
0}[q^{\pm 1}] [{\rm Im}(\rmat{\um'})].
\end{align*}
Thus our second assertion follows from the definition of
$\prec^\tb_{[\redez]}$; that is,
\begin{equation} \label{eq: good filter}
[\overset{\to}{S}_{[\redez]}(\um)] \in [{\rm Im}(\rmat{\um})] +
\displaystyle\sum_{ \substack{\um' \prec^{\tb}_{[\redez]} \um \\
\wt(\um')=\wt(\um)} } \Z_{\ge 0}[q^{\pm 1}] [{\rm Im}(\rmat{\um'})]. \qedhere
\end{equation}
\end{proof}

\begin{proposition} \label{prop: q-comm comparable}
Let $\al$ and $\be$ be in a sectional path of $\Upsilon_{[\redez]}$. Then $(\al,\be)$ is $[\redez]$-simple and we have
$$ S_{[\redez]}(\al) \conv S_{[\redez]}(\be) \simeq S_{[\redez]}(\be) \conv S_{[\redez]}(\al) \text{ is simple.} $$
\end{proposition}

\begin{proof} Proposition \ref{prop: sectional simplicity} implies that $(\al,\be)$ is a simple pair with respect to $\prec_{[\redez]}$. Thus our assertion follows from Theorem \ref{thm: app KLR1}.
\end{proof}

\begin{corollary}
Let $\beta_1, \beta_2, \cdots, \beta_p$ be in a sectional path of $\Upsilon_{[\redez]}$.  We have
\[S_{[\redez]}([\beta_1])^{\conv m_1}\conv\cdots \conv S_{[\redez]}(\beta_p)^{\conv m_{p}} \ \text{is simple for any $(m_1, m_2, \cdots, m_p)\in \Z_{\geq 0}^p$.} \]
\end{corollary}

\begin{remark}
By the works in \cite{BKM12,Kato12,Mc12}, $S_{\redez}(\beta)$'s categorify the dual PBW generators of $\g$ associated to $\redez$, which are also elements of the dual canonical basis.
Hence our results in this section tell that the dual PBW monomials {\it depend only on $[\redez]$} (up to $q^\Z$) and some of them are {\it $q$-commutative} under the circumstances we characterized.
In particular, when $R$ is {\it symmetric} and $\cor$ is of characteristic $0$, simple $R$-modules categorify the dual canonical basis (\cite{R11,VV09}). Hence \eqref{eq: good filter}
provides finer information on transition map between the dual canonical basis and the dual PBW basis associated to $[\redez]$.
\end{remark}

By \eqref{ref redez h}, one can observe the following similarity among $\{ S_{[\redez]}(\al) \}$ and
$\{ S_{[\redez']}(\al') \}$ for $[\redez],[\redez']$ in the same $r$-cluster point $\lf \redez \rf$:

\begin{corollary} \label{cor: similarity} For $[\redez]$ of $w_0$, let $(i_1,i_2,\ldots,i_k)$ be a sequence of indices such that
\[  i_k \text{ is a sink of } [\redez] \ r_{i_{1}} \cdots r_{i_{k-1}}.\]
Set $w = s_{i_{k-1}} \cdots  s_{i_1}$.  For $(\al,\be) \in (\Phi^+)^2$ with $[\redez]$-simple and $w \cdot \alpha, w \cdot \beta \in \Phi^+$, we have
$$ S_{[\redez] \cdot r_{\redex}}(w \cdot \al) \conv S_{[\redez] \cdot r_{\redex}}(w \cdot \be) \simeq S_{[\redez] \cdot r_{\redex}}(w \cdot \be) \conv S_{[\redez] \cdot r_{\redex}}(w \cdot \al) \text{ is simple,} $$
where $r_{\redex} \seteq r_{i_{1}} \cdots r_{i_{k-1}}$.
\end{corollary}

\appendix

\section{$r$-cluster points of $A_4$} \label{Sec:Appen_A4}

There are $62$ commutation classes of $w_0$ for $A_4$ (see
\cite[Table 1]{B99} and \cite[A006245]{OEIS}). We can check that the
$62$ commutation classes are classified into $3$-cluster points with respect to $\sigma = {}^*$ as
follows:
\begin{center}
$\begin{matrix}
\text{Type 1} \\
(5,5)
\end{matrix}$
\begin{tabular}{ | c |c | c |c |c|c| c | c | } \hline
A01 & 1213214321 & A02 & 2132143421 & A03 & 1214342312 & A04 &
3214342341 \\ \hline A05 & 4342341234 & A06 & 1321434231  & A07 &
2143423412 & A08 & 1434234123 \\ \hline
\end{tabular}
\end{center}

\begin{center}
$\begin{matrix}
\text{Type 2} \\
(4,6)
\end{matrix}$
\begin{tabular}{ | c |c | c |c |c|c| c | c | } \hline
B01 & 2123214321 & B02 & 1232143231 & B03 & 1232124321 & B04 &
1213243212 \\ \hline B05 & 2132314321 & B06 & 1323124321 & B07 &
1213432312 & B08 & 1323143231 \\ \hline B09 & 2321243421 & B10 &
2132434212 & B11 & 2124342312 & B12 & 1243421232 \\ \hline B13 &
3231243421 & B14 & 2321432341 & B15 & 2134323412 & B16 & 2143234312
\\ \hline B17 & 3212434231 & B18 & 1324342123 & B19 & 1243423123 &
B20 & 1432341232 \\ \hline B21 & 3214323431 & B22 & 1343234123 & B23
& 1432343123 & B24 & 2434212342 \\ \hline B25 & 3243421234 & B26 &
2434231234 & B27 & 4323412342 & B28 & 4342123423 \\ \hline B29 &
3432341234 & B30 & 4323431234 & B31 & 4342312343 & B32 & 3231432341
\\ \hline
\end{tabular}
\end{center}

\begin{center}
$\begin{matrix}
\text{Type 3} \\
(3,7)
\end{matrix}$
\begin{tabular}{ | c |c | c |c |c|c| c | c | } \hline
C01 & 2123243212 & C02 & 2321234321 & C03 & 2132343212 & C04 &
2123432312 \\ \hline C05 & 3212324321 & C06 & 1232432123 & C07 &
1234321232 & C08 & 3231234321 \\ \hline C09 & 3212343231 & C10 &
1323432123 & C11 & 1234323123 & C12 & 3234321234 \\ \hline C13 &
2324321234 & C14 & 2343212342 & C15 & 2432123432 & C16 & 4321234232
\\ \hline C17 & 3432312343 & C18 & 2343231234 & C19 & 4323123432 &
C20 & 3243212343 \\ \hline C21 & 3432123423 & C22 & 4321234323 & & &
& \\ \hline
\end{tabular}
\end{center}

\bibliographystyle{amsplain}

\begin{thebibliography}{99}

\bibitem{ARS} M. Auslander, I. Reiten and S. Smalo, {\it Representation theory of Artin algebras}, Cambridge studies in advanced
mathematics {\bf 36}, Cambridge 1995.

\bibitem{ASS} I. Assem, D. Simson and A. Skowro\'{n}ski, {\it Elements of the representation theory of associative algebras. Vol.1}, London
Math. Soc. Student Texts {\bf 65}, Cambridge 2006.


\bibitem{B99}
R. Bedard, {\it On commutation classes of reduced words in Weyl
groups}, European J. Combin. {\bf 20} (1999), 483--505.

\bibitem{Birkhoff}
G. Birkhoff, {\em Lattice Theory}, Amer. Math. Soc., Providence, 1967.

\bibitem{Bour}
N.~Bourbaki.
\newblock {\em \'{E}l\'ements de math\'ematique. {F}asc. {XXXIV}. {G}roupes et
  alg\`ebres de {L}ie. {C}hapitres {IV}--{VI}}.
\newblock Actualit\'es Scientifiques et Industrielles, No. 1337. Hermann,

\bibitem{BKM12}
J. Brundan, A. Kleshchev and P. J. McNamara, {\it Homological
properties of finite Khovanov-Lauda-Rouquier algebras}, Duke Math.
J., {\bf 163} (2014), 1353--1404.

\bibitem{CT15}
J. Claxton and P. Tingley, {\it Young tableaux, multisegments, and PBW bases}, S\'{e}m. Lothar.
Combin. 73 (2015), Article B73c

\bibitem{Gab80} P. Gabriel, {\it Auslander-Reiten sequences and Representation-finite algebras}, Lecture notes in
Math., vol. 831, Springer-Verlag, Berlin and New York, (1980),
pp.1-71.

\bibitem{HL11} D. Hernandez and B. Leclerc, {\it Quantum Grothendieck rings and derived Hall algebras},
arXiv:1109.0862v2 [math.QA], J. Reine Angew. Math. {\bf 701} (2015), 77--126.

\bibitem{KKK13b}
S.-J. Kang, M. Kashiwara and M. Kim, {\it Symmetric quiver Hecke
algebras and R-matrices of quantum affine algebras II}, Duke Math. J. {\bf 164}(8), 1549--1602.

\bibitem{Kato12} S. Kato,
{\it Poincar\'e-Birkhoff-Witt bases and Khovanov-Lauda-Rouquier
algebras}, Duke Math. J. \textbf{163}, 3 (2014), 619--663.

\bibitem{KL09}
M.~Khovanov and A. D.~Lauda, \emph{A diagrammatic approach to
categorification of quantum groups
  {I}}, Represent. Theory \textbf{13} (2009), 309--347.


\bibitem{Krause}
H. Krause, {\em Representations of quivers via reflection functors}, arXiv:0804.1428 (2008).


\bibitem{Lus90} G. Lusztig {\it Canonical Bases Arising from Quantized Enveloping Algebras}, J. Amer. Math. Soc.
Vol. 3, No. 2, (1990), 447-498.

\bibitem{Lus90A} \bysame, {\it Quantum groups at roots of $1$}, Geom. Dedicata,
{\bf 33} (1990), 89-113.

\bibitem{Lus93} \bysame, {\it Introduction to Quantum Groups}, Birkh\"auser, 1993.

\bibitem{Lus97} \bysame, {\em ``Canonical bases and Hall algebras" in Representation Theories and
Algebraic Geometry (Montreal, PQ, 1997)}, NATO Adv. Sci. Inst. Ser. C Math.
Phys. Sci. {\bf 514}, Kluwer Acad. Publ., Dordrecht, (1998), 365--399.


\bibitem{Mc12}
P. McNamara, {\it Finite dimensional representations of
Khovanov-Lauda-Rouquier algebras I: finite type}, J. Reine Angew. Math. {\bf 707} (2015), 103--124.

\bibitem{OEIS}
N.~J.~A. Sloane, {\it The On-Line Encyclopedia of Integer Sequences}, published electronically
at http://oeis.org.

\bibitem{Oh14A}
S.-j. Oh, {\it Auslander-Reiten quiver of type A and generalized
quantum affine Schur-Weyl duality}, Trans. Amer. Math. Soc. {\bf 369} (2017), 1895--1933

\bibitem{Oh14D} \bysame,
{\em Auslander-Reiten quiver of type D and generalized quantum affine Schur-Weyl duality}, J. Algebra {\bf 460} (2016), 203-252.

\bibitem{Oh15E}
\bysame, {\it Auslander-Reiten quiver and representation theories
related to KLR-type Schur-Weyl duality}, arXiv:1509.04949.

\bibitem{Papi94}
P. Papi, {\it A characterization of a special ordering in a root
system}, Proc. Amer. Math. {\bf120}
 (1994), 661--665.

\bibitem{R80} C.M. Ringel, {\it Tame algebras}, Proceedings ICRA 2, Springer LNM 831, (1980), 137--87.

\bibitem{R96}
\bysame, {\it PBW-bases of quantum groups}, J. Reine Angew. Math.
{\bf 470} (1996), pp. 51--88.

\bibitem{R08}
R.~Rouquier, \emph{2 {K}ac-{M}oody algebras}, arXiv:0812.5023
(2008).

\bibitem{R11}
\bysame, {\em Quiver Hecke algebras and 2-Lie algebras}, Algebra
Colloq. {\bf 19} (2012), no. 2, 359--410.

\bibitem{Saito94}
Y. Saito, {\em PBW basis of quantized universal enveloping algebras}, Publ. RIMS, Kyoto Univ. {\bf 30} (1994), 209--232.

\bibitem{S14}
R.~Schiffler, {\em Quiver Representations}, CMS Books in Mathematics, Springer Verlag, 2014.

\bibitem{VV09}
M. Varagnolo and E. Vasserot,
 \emph{Canonical bases and KLR algebras},
J. reine angew. Math. \textbf{659} (2011), 67--100.


\bibitem{Zelo87}
D.P. Zelobenko, {\it Extremal cocyles on Weyl groups}, Funkz. Analiz i ego Pril. {\bf 21} (1987), 11--21

\end{thebibliography}

\end{document}